\documentclass[preprint,3p,compress,12pt]{elsarticle}

\usepackage{amscd}
\usepackage{amssymb}
\usepackage{amsmath}
\usepackage{graphicx,multirow}
\usepackage{array}
\usepackage{subfigure }
\usepackage[misc]{ifsym}
\usepackage{amsfonts}
\usepackage{latexsym}
\usepackage{color}
\usepackage{booktabs}
\usepackage{algorithm}
\usepackage{algorithmicx}
\usepackage{hyperref}

\usepackage{placeins}

\numberwithin{equation}{section}
\newtheorem{expl}{Example}[section]

\newtheorem{theorem}{Theorem}[section]
\newtheorem{lemma}[theorem]{Lemma}

\newtheorem{rem}[theorem]{Remark}

 \numberwithin{figure}{section}

\begin{document}
\begin{frontmatter}

\title{\textbf{Full- and low-rank exponential Euler integrators for the Lindblad equation}}
\author[1]{Hao Chen}
\ead{hch@cqnu.edu.cn}
\author[2]{Alfio Borz\`{i}}
\ead{alfio.borzi@mathematik.uni-wuerzburg.de}
\author[3]{Denis Jankovi\'{c}}
\ead{denis.jankovic@ipcms.unistra.fr}
\author[3]{Jean-Gabriel Hartmann}
\ead{jeangabriel.hartmann@ipcms.unistra.fr}
\author[3]{Paul-Antoine Hervieux}
\ead{paul-antoine.hervieux@ipcms.unistra.fr}

\address[1]{College of Mathematics Science, Chongqing Normal University, Chongqing, China.}
\address[2]{Institut f\"{u}r Mathematik,
Universit\"{a}t W\"{u}rzburg, W\"{u}rzburg, Germany.}
\address[3]{Institut de Physique et Chimie des Mat\'{e}riaux de Strasbourg (IPCMS), Universit\'{e} de Strasbourg, Strasbourg, France.}

\begin{abstract}
The Lindblad equation is a widely used quantum master equation to model the dynamical evolution of open quantum systems whose states are described by density matrices. These solution matrices are characterized
by semi-positiveness and trace preserving properties, which must be guaranteed
in any physically meaningful numerical simulation.
In this paper, novel full- and low-rank exponential Euler integrators are developed for approximating the Lindblad equation that preserve positivity and trace unconditionally.
Theoretical results are presented that provide sharp error estimates for
the two classes of exponential integration methods.
Results of numerical experiments are discussed that illustrate the
effectiveness of the proposed schemes, beyond present state-of-the-art
capabilities.
\end{abstract}

\end{frontmatter}

\section{Introduction}
The Lindblad equation is a widely used Markovian quantum master equation to model the dynamical evolution of open quantum systems \cite{Breuer,Davies1976}. The quantum master equation plays a crucial role in many application areas including condensed matter physics, quantum physics and quantum information theory. In this paper, we consider quantum systems consisting of $K$ dephasing $d$-level qudits undergoing Markovian open quantum dynamics. The time dynamics of these systems is given by the Gorini-Kossakowski-Sudarshan-Lindblad equation (or Lindblad equation for short) \cite{Gorini,Lindblad}:
\begin{equation}\label{equ1.1}
  \dot{\rho}(t)=-i[H,\rho(t)]+\sum_{k=1}^K\gamma_k\left(L_k\rho(t) L_k^{\dag}-\frac{1}{2}\left\{L_k^{\dag}L_k,\rho(t)\right\}\right),
\end{equation}
where $\rho(t)\in \mathbb{C}^{m\times m}$ is the density matrix describing the state of the system, $H=H^{\dag}$ is the Hamiltonian operator describing the unitary evolution of the qudit, $L_k$ are the Lindblad or jump operators characterizing the dissipation channels, and $\gamma_k\geq0$ are the decay parameters for each of the $K$ channels.

The Lindblad equation can be also written in a vectorized form,
using Dirac's bracket notation, as follows:
\begin{equation}\label{equ1.2}
  |\dot{\rho}(t)\rangle=\mathcal{L}|\rho(t)\rangle,
\end{equation}
with
\begin{equation*}
  \mathcal{L}=-i I\otimes H+i H^\top\otimes I+\sum_{k=1}^K\gamma_k
  \left(L_k^*\otimes L_k-\frac{1}{2}I\otimes (L_k^{\dag}L_k)-\frac{1}{2}(L_k^\top L_k^*)\otimes I\right),
\end{equation*}
where $|\rho\rangle$ is the vectorized density matrix and $\mathcal{L}\in \mathbb{C}^{m^2\times m^2}$ is the matrix representation of the Lindblad generator. The superscripts $\dag$, $\top$ and $*$ denote the adjoint, transpose and complex conjugate operators, respectively.

The Lindblad equation possesses two important properties \cite{Gorini,Lindblad}:
it is positivity and trace preserving, which means that the solution $\rho(t)$ is Hermitian positive semidefinite and has unit trace if the initial data is Hermitian positive semidefinite with unit trace. These properties of the density
matrix are of fundamental physical significance, and whether they
can be preserved at the discrete level is a significant issue in numerical simulations, and especially for long-time simulation \cite{Riesch}. The major goal of this work is to develop positivity and trace preserving numerical schemes for approximating the Lindblad equation \eqref{equ1.1} with guaranteed accuracy.

The literature on numerical methods for solving the Lindblad equation is
relatively scarce. Explicit Runge-Kutta methods have been discussed in \cite{Riesch}, where the authors show that these methods do not preserve the positivity of the density matrix. The Crank-Nicolson method has been studied in \cite{Ziolkowski} for a two-level Lindblad equation and it is shown in \cite{Bidegaray} that the Crank-Nicolson scheme can not preserve the positivity in general. In fact, standard Runge-Kutta methods and multistep methods cannot be better than first-order accurate if they must
preserve positivity; see, e.g., \cite{Blanes2022,Martiradonna2020}. On the other hand, there are methods based on matrix exponential \cite{Riesch,Riesch1,Saut} for solving the vector representation \eqref{equ1.2} of the Lindblad equations. This kind of methods are positivity and trace preserving since the differential equation \eqref{equ1.2} is solved analytically. However, the coefficient matrix $\mathcal{L}$ is of size $m^2\times m^2$ and the computation of the product of the associated matrix exponential with vectors is expensive if $m^2$ is large. Further, there are methods based on operator splitting techniques \cite{Songolo1,Songolo2} and low-rank operator splitting schemes \cite{LeBris1,LeBris2}, where by means of a normalization procedure the low-rank splitting schemes preserve positivity and trace; however, rigorous error analysis is still unavailable. Moreover, Kraus representation approximation methods have been studied in \cite{Cao}, where it is shown that these methods are positivity and trace preserving,
where the latter is achieved by a normalization procedure. In addition, a decomposition of nonunitary operators approach has been developed in \cite{Schlimgen} for the vectorized Lindblad equation.

Besides the above mentioned approaches, several numerical methods have been studied for large-scale problems \cite{Weimer,Werner}, including methods based on matrix product states \cite{Orus,Schollwock,Vidal1,Vidal2} and strategies based on neural networks \cite{Hartmann,Nagy,Reh}. Tree tensor networks \cite{Arceci,Sulz} have also been developed to simulate quantum many-body states.

Noting that, in general, positivity and trace preserving schemes are relatively less considered and rigorous error analysis on numerical integrators for the Lindbald equation is largely unavailable, the aim of this paper is to develop and analyze positivity and trace preserving numerical schemes for solving the Lindblad equation \eqref{equ1.1}. Inspired by the exponential integrators for solving differential matrix Riccati equations \cite{Chen1,Chen2}, we present novel full- and low-rank exponential Euler integrators for the Lindblad equation. We prove that our exponential Euler schemes preserve positivity and trace unconditionally, that is, the new methods are unconditionally stable. However, while the full-rank scheme
does not require any normalization procedure to preserve unit trace, in the low-rank
case normalization is still required. On the other hand, trace preservation
in the full-rank case allows us to prove sharp accuracy estimates that
are also instrumental for obtaining the accuracy estimates for the low-rank
method presented in this paper. Clearly, one disadvantage of the low-rank approach
is to introduce an additional approximation error due to the low-rank representation
of the density matrix. On the other hand, we have the benefit of drastically reducing computational costs as we show in the section of numerical experiments.

This paper is organized as follows. In Section 2, we introduce our full-rank exponential Euler integrator for discretizing the Lindblad equation and discuss related implementation issues. In Section 3, we show that the proposed full-rank exponential scheme preserves positivity and unit trace of the solution unconditionally. Section 4 is dedicated to the error analysis of the full-rank exponential Euler integrator. In Section 5, we propose the low-rank exponential Euler scheme and provide its error estimate. In Section 6, many numerical experiments are performed for the Lindblad equation with various Hamiltonians, including the tests of convergence rate, positivity and unit trace preservation, and effectiveness of the low-rank algorithm. Also in this section,
we compare our exponential Euler schemes with the Lindblad solvers of the
well-known QuTip framework for simulation of open quantum systems \cite{Johansson}, showing that, while the latter may provide higher accuracy, they seem not able
to provide the required positivity property, which is guaranteed by the proposed
full- and low-rank exponential Euler schemes. Moreover, it is shown that
the new low-rank exponential scheme performs much faster than
all other schemes considered in this paper, as the dimensionality of the problem increases.
A section of conclusion completes this work.

\section{An exponential Euler integrator}
In this section, we develop an exponential integrator to solve the Lindblad equation \eqref{equ1.1} on the time interval $[0,T]$, for some $T>0$, equipped with the initial condition
$\rho(0)=\rho_0$, where $\rho_0$ is a Hermitian and positive semidefinite matrix with unit trace.
First, let us define
\begin{equation}\label{equ2.1}
  A:=- i \, H-\frac{1}{2}\sum_{k=1}^K\gamma_k \, L_k^{\dag} \, L_k.
\end{equation}
Then, we can rewrite \eqref{equ1.1} as follows:
\begin{equation}\label{equ2.2}
  \dot{\rho}=A \, \rho+\rho \,A^{\dag}+\sum_{k=1}^K\gamma_k \, L_k \, \rho \,L_k^{\dag}.
\end{equation}
Given a positive integer $N$, we divide the time interval by $N$
subintervals of size $\tau=\frac{T}{N}$, with endpoints
$t_n=n\tau$, $n=0,1,\ldots,N$. Then, integrating the equation \eqref{equ2.2} from $t_n$ to $t_{n+1}$ and using the variation-of-constants formula yields
\begin{equation}\label{equ2.3}
  \rho(t_{n+1})=e^{\tau A}\rho(t_n)e^{\tau A^{\dag}}+\sum_{k=1}^K\gamma_k\int_0^{\tau}e^{(\tau-s)A}L_k \, \rho(t_n+s) \, L_k^{\dag}e^{(\tau-s)A^{\dag}}ds.
\end{equation}
Applying this formula again for $\rho(t_n+s)$ within the integral, we obtain
\begin{eqnarray}\label{equ2.4}
  \rho(t_{n+1})=e^{\tau A} \rho(t_n) \, e^{\tau A^{\dag}}+\sum_{k=1}^K\gamma_k
  \int_0^{\tau}e^{(\tau-s)A}L_k \, e^{sA} \, \rho(t_n) \, e^{sA^{\dag}} \, L_k^{\dag} \, e^{(\tau-s)A^{\dag}}ds+R_{n,1},
\end{eqnarray}
where
\begin{equation}\label{equ2.4b}
  R_{n,1}=\sum_{k=1}^K\sum_{j=1}^K\gamma_k\gamma_j\int_0^{\tau}e^{(\tau-s)A}L_k
  \left(\int_0^se^{(s-v)A}L_j \, \rho(t_n+v) \, L_j^{\dag} \, e^{(s-v)A^{\dag}}dv\right)
  L_k^{\dag} \, e^{(\tau-s)A^{\dag}}ds.
\end{equation}
By truncating the term $R_{n,1}$ and approximating the matrix exponentials by $e^{(\tau-s)A}\approx I$ and $e^{(\tau-s)A^{\dag}}\approx I$ in \eqref{equ2.4}, we obtain our full-rank exponential Euler (FREE) integrator, for the approximation $\{\rho_n\}_{n=0}^N$ to $\{\rho(t_n)\}_{n=0}^N$, as follows:
\begin{equation}\label{equ2.5}
  \rho_{n+1}=e^{\tau A} \rho_n \, e^{\tau A^{\dag}}+\sum_{k=1}^K\gamma_k
  \int_0^{\tau}L_k \, e^{sA} \, \rho_n \, e^{sA^{\dag}} L_k^{\dag}\, ds ,
\end{equation}
for $0\leq n\leq N-1$.

\begin{rem}
  If we make the approximation $\rho(t_n+s)\approx \rho_n$ within the integral in \eqref{equ2.3}, then we obtain the ``standard" exponential Euler scheme for solving \eqref{equ1.1} given by
  \begin{eqnarray}\label{equ2.6}
\rho_{n+1}&=&e^{\tau A}\rho_n \, e^{\tau A^{\dag}}+\sum_{k=1}^K\gamma_k\int_0^{\tau}e^{(\tau-s)A}L_k \, \rho_n \, L_k^{\dag} \, e^{(\tau-s)A^{\dag}}ds \nonumber\\
&=&e^{\tau A}\rho_n \, e^{\tau A^{\dag}}+\sum_{k=1}^K\gamma_k\int_0^{\tau}e^{sA}L_k \, \rho_n \, L_k^{\dag}e^{sA^{\dag}}ds.
  \end{eqnarray}
Note that the ``standard" exponential Euler scheme \eqref{equ2.6} can preserve the positive semidefinite property and is first-order convergent. However, this scheme can not preserve unit trace of the Lindblad equation. Similar ``standard" exponential integrators have been proposed and analyzed for solving differential matrix Riccati equations in \cite{Chen1,Chen2}. Such matrix-valued exponential integrators are natural extensions of the vector-valued exponential integrators, for which we refer to the review \cite{Hochbruck} for details.
\end{rem}

\begin{rem}
  We remark that the exponential Euler integrator \eqref{equ2.5} can be extended straightforwardly to solve Lindblad equations with time-dependent Hamiltonians $H(t)$, in which case the exponential Euler scheme reads
   \begin{equation}\label{equ2.6b}
  \rho_{n+1}=e^{\tau A_n} \rho_n \, e^{\tau A_n^{\dag}}+\sum_{k=1}^K\gamma_k
  \int_0^{\tau}L_ke^{sA_n}\rho_n \, e^{sA_n^{\dag}} L_k^{\dag}ds,
\end{equation}
where
\begin{equation*}
  A_n=-iH(t_n)-\frac{1}{2}\sum_{k=1}^K\gamma_kL_k^{\dag}L_k.
\end{equation*}
Since the analysis of this scheme is similar to the one with time-independent Hamiltonian, for ease of presentation, we will only present numerical analysis for the FREE scheme \eqref{equ2.5}.
\end{rem}

Now, we briefly illustrate the practical implementation of the proposed exponential Euler scheme \eqref{equ2.5}. Define
\begin{equation*}
  W_n=\int_0^{\tau}e^{sA} \, \rho_n  \, e^{sA^{\dag}}ds.
\end{equation*}
By using the properties of Kronecker products, we can show that $W_n$ is the solution of the following algebraic Lyapunov equation (see \cite{Chen1,Chen2} for details)
\begin{equation*}
  AW_n+W_nA^{\dag}=e^{\tau A}\rho_ne^{\tau A^{\dag}}-\rho_n.
\end{equation*}
It follows that the exponential Euler integrator \eqref{equ2.5} is equivalent to
\begin{subequations}\label{equ2.7}
  \begin{align}
  &AW_n+W_nA^{\dag}=e^{\tau A} \, \rho_n \, e^{\tau A^{\dag}}-\rho_n,\label{equ2.7a}\\
  &\rho_{n+1}=e^{\tau A} \, \rho_n  \, e^{\tau A^{\dag}}+\sum_{k=1}^K\gamma_k \, L_k \, W_n \, L_k^{\dag}.\label{equ2.7b}
\end{align}
\end{subequations}
Note that at every time step of the exponential Euler integrator \eqref{equ2.7}, we need to solve an algebraic Lyapunov equation and compute $2m$ matrix-vector products associated with matrix exponential $e^{\tau A}$.

\section{Positivity and trace preserving properties}
In this section, we study the positivity and trace preserving properties of the exponential Euler integrator \eqref{equ2.5}. Throughout the paper, we use the following notation.
For any matrix $\varrho\in \mathbb{C}^{m\times m}$, we denote $\mbox{Tr}(\varrho)$ the trace of the matrix $\varrho$. For the trace, we have the cyclic property that $\mbox{Tr}(BCD)=\mbox{Tr}(CDB)=\mbox{Tr}(DBC)$. We write $\varrho\geq0$ if the matrix $\varrho$ is Hermitian and positive semidefinite.

For any matrix $\varrho\in \mathbb{C}^{m\times m}$, the trace norm, or the Schatten-1 norm, is defined as $\|\varrho\|_1=\mbox{Tr}(\sqrt{\varrho^{\dag}\varrho})=\sum_{j=1}^m\sigma_j(\varrho)$, where $\sigma_1(\varrho)\geq\sigma_2(\varrho)\geq\cdots\geq\sigma_m(\varrho)$ are the singular values of $\varrho$. If $\varrho\geq0$, we have that $\|\varrho\|_1=\mbox{Tr}(\varrho)$. We denote $\|\varrho\|_F$ the Frobenius norm of $\varrho$.

Now, we study the positivity preserving property of the proposed FREE integrator.
\begin{theorem}\label{thm3.1}
Assume that $\rho_0$ is Hermitian and positive semidefinite. Then, for any time step size $\tau>0$, the solution of the full-rank exponential Euler scheme \eqref{equ2.5} given by $\{\rho_n\}_{n=0}^N$, is Hermitian and positive semidefinite.
\end{theorem}
\begin{proof}
  We prove the theorem by induction. Assume that the result holds for $n=j$, i.e. $\rho_j$ is Hermitian and positive semidefinite. Next, we prove that the result holds for $n=j+1$. Since $\rho_j$ is Hermitian and positive semidefinite, we see that both terms on the right-hand side of \eqref{equ2.5} are Hermitian and positive semidefinite. It follows that $\rho_{j+1}$ is also Hermitian and positive semidefinite, which completes the proof.
\end{proof}

The following result will play a key role in the analysis of the trace preserving property of the exponential Euler scheme.

\begin{lemma}\label{lem3.1}
  For any matrix $\varrho\in \mathbb{C}^{m\times m}$ and any $t\geq0$, it holds that
\end{lemma}
\begin{equation}\label{equ3.1}
    \mbox{Tr}\left(e^{t A} \, \varrho \, e^{t A^{\dag}}+\sum_{k=1}^K\gamma_k
  \int_0^{t}L_k \, e^{sA} \, \varrho \, e^{sA^{\dag}} L_k^{\dag} \, ds\right)=\mbox{Tr}(\varrho).
  \end{equation}
\begin{proof}
  Note that the solution of the differential equation
  \begin{equation}\label{equ3.2}
    \dot{\sigma}=A \, \sigma+\sigma \, A^{\dag},~~~\sigma(0)=\varrho ,
  \end{equation}
  is given by
  \begin{equation}\label{equ3.3}
    \sigma(t)=e^{tA} \, \varrho \, e^{tA^{\dag}}.
  \end{equation}
  Using the definition of $A$ in \eqref{equ2.1}, we obtain that the differential equation \eqref{equ3.2} is equivalent to
  \begin{equation*}
    \dot{\sigma}=-i H\sigma+i\sigma H-\frac{1}{2}\sum_{k=1}^K\gamma_k\left(L_k^{\dag}L_k \, \sigma+\sigma \, L_k^{\dag}L_k\right),~~~\sigma(0)=\varrho.
  \end{equation*}
  By the variation-of-constants formula, we have
  \begin{equation*}
    \sigma(t)=e^{-itH}\varrho \, e^{itH}-\frac{1}{2}\sum_{k=1}^K\gamma_k\int_0^te^{-i(t-s)H}\left(L_k^{\dag}L_k \, \sigma(s)+\sigma(s) \, L_k^{\dag}L_k\right) e^{i(t-s)H} \, ds.
  \end{equation*}
  Using the cyclic property of the trace and \eqref{equ3.3},
   we obtain
  \begin{eqnarray}\label{equ3.4}
    \mbox{Tr}(e^{tA}\varrho e^{tA^{\dag}})&=&\mbox{Tr}(\sigma(t))=\mbox{Tr}(e^{-itH}\varrho e^{itH})\nonumber\\
    &&-\frac{1}{2}\sum_{k=1}^K\gamma_k\int_0^t\mbox{Tr}\left(e^{-i(t-s)H}\left(L_k^{\dag}L_k\sigma(s)+\sigma(s) L_k^{\dag}L_k\right) e^{i(t-s)H}\right)ds\nonumber\\
    &=&\mbox{Tr}(\varrho)-\frac{1}{2}\sum_{k=1}^K\gamma_k\int_0^t\mbox{Tr}\left(L_k^{\dag}L_k\sigma(s)+\sigma(s) L_k^{\dag}L_k \right)ds\nonumber\\
    &=&\mbox{Tr}(\varrho)-\sum_{k=1}^K\gamma_k\int_0^t\mbox{Tr}\left(L_k^{\dag}L_k\sigma(s)\right)ds\nonumber\\
    &=&\mbox{Tr}(\varrho)-\sum_{k=1}^K\gamma_k\int_0^t\mbox{Tr}\left(L_k\sigma(s)L_k^{\dag}\right)ds\nonumber\\
    &=&\mbox{Tr}(\varrho)-\sum_{k=1}^K\gamma_k\int_0^t\mbox{Tr}\left(L_k \, e^{sA}\varrho \, e^{sA^{\dag}} L_k^{\dag}\right)ds.
  \end{eqnarray}
  Then, the desired result \eqref{equ3.1} follows.
\end{proof}

Now, we are in the position to give the following theorem.

\begin{theorem}\label{thm3.2}
  Assume that the initial data satisfies $\mbox{Tr}(\rho_0)=1$. Then, for any time step size $\tau>0$, the full-rank exponential Euler integrator \eqref{equ2.5} preserves unit trace
  in the sense that
  \[\mbox{Tr}(\rho_n)=1,~~~\, n\geq0.\]
\end{theorem}
\begin{proof}
  We prove the theorem by induction. We have $\mbox{Tr}(\rho_0)=1$. Assume that the result holds for $n=j$, i.e., $\mbox{Tr}(\rho_j)=1$. Next we prove the result holds for $n=j+1$. Using \eqref{equ2.5} and applying Lemma \ref{lem3.1}, we have $\mbox{Tr}(\rho_{j+1})=\mbox{Tr}(\rho_j)$, which completes the proof.
\end{proof}

\begin{rem}
Note from Theorems \ref{thm3.1} and \ref{thm3.2} that $\|\rho_n\|_1=\mbox{Tr}(\rho_n)=1$ for any $n\geq0$ and any $\tau>0$, which indicates that the proposed exponential Euler integrator \eqref{equ2.5} is unconditionally stable.
\end{rem}

\section{Accuracy of approximation}
In this section, we investigate the convergence of the exponential Euler scheme \eqref{equ2.5}. First, we present some preparatory results that will be useful in the error analysis.

\begin{lemma}\label{lem4.2}
  For any matrix $B\in \mathbb{C}^{m\times m}$ and any Hermitian matrix $\varrho\in \mathbb{C}^{m\times m}$, it holds that
  \[\left\|B\varrho B^{\dag}\right\|_1\leq \left\|B|\varrho|B^{\dag}\right\|_1,\]
  where $|\varrho|=\sqrt{\varrho^{\dag}\varrho}$.
\end{lemma}
\begin{proof}
  Note that for any Hermitian matrix $\varrho$, we have the spectral decomposition $\varrho=QD Q^{\dag}$, where $Q$ is unitary and $D$ is diagonal. Note that $|\varrho|=Q|D|Q^{\dag}$. So we only need to prove that the following inequality
  \[\left\|PDP^{\dag}\right\|_1\leq \left\|P|D|P^{\dag}\right\|_1\]
  holds for any $P=[p_1,\ldots,p_m]\in \mathbb{C}^{m\times m}$ and any diagonal matrix $D=\mbox{diag}(d_1,\ldots,d_m)\in \mathbb{R}^{m\times m}$.

  Since $PDP^{\dag}$ is Hermitian, we have the spectral decomposition $PDP^{\dag}=U\Lambda U^{\dag}$, where $U=[u_1,\ldots,u_m]$ is unitary and $\Lambda=\mbox{diag}(\lambda_1,\ldots,\lambda_m)$ is the diagonal matrix containing eigenvalues of $PDP^{\dag}$. Using $\lambda_j=u_j^{\dag}PDP^{\dag}u_j$, $PDP^{\dag}=\sum\limits_{k=1}^md_kp_kp_k^{\dag}$, and $P|D|P^{\dag}=\sum\limits_{k=1}^m|d_k|p_kp_k^{\dag}$, we get
  \begin{eqnarray*}
    \left\|PDP^{\dag}\right\|_1&=&\sum_{j=1}^m\sigma_j(PDP^{\dag})=\sum_{j=1}^m|\lambda_j|=\sum_{j=1}^m\left|u_j^{\dag}PDP^{\dag}u_j\right|\\
    &=&\sum_{j=1}^m\left|u_j^{\dag}\left(\sum_{k=1}^md_kp_kp_k^{\dag}\right)u_j\right|\\
    &=&\sum_{j=1}^m\left|\sum_{k=1}^md_k (p_k^{\dag}u_j)^{\dag}(p_k^{\dag}u_j)\right|\\
    &\leq&\sum_{j=1}^m\sum_{k=1}^m|d_k| (p_k^{\dag}u_j)^{\dag}(p_k^{\dag}u_j)\\
    &=&\sum_{j=1}^mu_j^{\dag}\left(\sum_{k=1}^m|d_k|p_kp_k^{\dag}\right)u_j\\
    &=&\sum_{j=1}^mu_j^{\dag}P|D|P^{\dag}u_j=\mbox{Tr}(U^{\dag}P|D|P^{\dag}U)\\
    &=&\mbox{Tr}(P|D|P^{\dag})=\left\|P|D|P^{\dag}\right\|_1,
  \end{eqnarray*}
  and the required result follows.
\end{proof}

\begin{lemma}\label{lem4.1}
   For any Hermitian matrix $\varrho\in \mathbb{C}^{m\times m}$, it holds that
  \begin{equation*}
    \left\|e^{tA} \, \varrho \, e^{tA^{\dag}}\right\|_1\leq \left\|\varrho\right\|_1,~~~\, t\geq0.
  \end{equation*}
\end{lemma}
\begin{proof}
 Let $|\varrho|=\sqrt{\varrho^{\dag}\varrho}$ and we note that $|\varrho|\geq0$ and $\|\varrho\|_1=\mbox{Tr}(|\varrho|)$.
 For any $t\geq0$, applying Lemma \ref{lem4.2} and Lemma \ref{lem3.1} yields
  \begin{eqnarray*}
    \left\|e^{tA}\varrho e^{tA^{\dag}}\right\|_1&\leq& \left\|e^{tA}|\varrho| e^{tA^{\dag}}\right\|_1 \\
    &=&\mbox{Tr}\left(e^{tA}|\varrho| e^{tA^{\dag}}\right)\\
    &=&\mbox{Tr}(|\varrho|)-\sum_{k=1}^K\gamma_k\int_0^t\mbox{Tr}\left(L_ke^{sA}|\varrho| e^{sA^{\dag}} L_k^{\dag}\right)ds\\
    &=&\|\varrho\|_1-\sum_{k=1}^K\gamma_k\int_0^t\left\|L_ke^{sA}|\varrho| e^{sA^{\dag}} L_k^{\dag}\right\|_1ds\\
    &\leq&\|\varrho\|_1.
  \end{eqnarray*}
  The conclusion follows.
\end{proof}

\begin{lemma}\label{lem4.3}
  For any Hermitian matrix $\varrho\in \mathbb{C}^{m\times m}$ and any $t\geq0$, it holds that
\end{lemma}
\begin{equation}\label{equ4.1}
    \left\|e^{t A}\varrho \, e^{t A^{\dag}}+\sum_{k=1}^K\gamma_k
  \int_0^{t}L_k \, e^{sA}\varrho \, e^{sA^{\dag}} L_k^{\dag} \, ds\right\|_1\leq\|\varrho\|_1.
  \end{equation}
\begin{proof}
  Using Lemma \ref{lem4.2} and noting that $|\varrho|\geq0$, we have
  \begin{eqnarray*}
  &&\left\|e^{t A}\varrho e^{t A^{\dag}}+\sum_{k=1}^K\gamma_k
  \int_0^{t}L_ke^{sA}\varrho e^{sA^{\dag}} L_k^{\dag}ds\right\|_1\\
  &\leq& \left\|e^{t A}\varrho e^{t A^{\dag}}\right\|_1+\sum_{k=1}^K\gamma_k
  \int_0^{t}\left\|L_ke^{sA}\varrho e^{sA^{\dag}} L_k^{\dag}\right\|_1ds\\
  &\leq& \left\|e^{t A}|\varrho| e^{t A^{\dag}}\right\|_1+\sum_{k=1}^K\gamma_k
  \int_0^{t}\left\|L_ke^{sA}|\varrho| e^{sA^{\dag}} L_k^{\dag}\right\|_1ds\\
  &=&\mbox{Tr}\left(e^{t A}|\varrho| e^{t A^{\dag}}\right)+\sum_{k=1}^K\gamma_k
  \int_0^{t}\mbox{Tr}\left(L_ke^{sA}|\varrho| e^{sA^{\dag}} L_k^{\dag}\right)ds\\
  &=&\mbox{Tr}\left(e^{t A}|\varrho| e^{t A^{\dag}}+\sum_{k=1}^K\gamma_k
  \int_0^{t}L_ke^{sA}|\varrho| e^{sA^{\dag}} L_k^{\dag}ds\right)\\
  &=&\mbox{Tr}(|\varrho|),
  \end{eqnarray*}
  where the last equality follows from Lemma \ref{lem3.1}.
  The desired result follows since $\mbox{Tr}(|\varrho|)=\|\varrho\|_1$.
\end{proof}

For the discussion that follows, we define
\[C_1:=\sum_{k=1}^K\gamma_k \, \|L_k\|_1^2,~~~C_2:=\|A\|_1.\]
Now, we are ready to prove our error estimate of the full-rank exponential Euler scheme \eqref{equ2.5}. We have
\begin{theorem}\label{thm4.1}
  Let $\rho(t)$ be the solution of the Lindblad equation \eqref{equ1.1} with
  $\rho(0)=\rho_0$, and $\{\rho_n\}_{n=0}^N$ be the corresponding numerical solution generated by the full-rank exponential Euler scheme \eqref{equ2.5}. Then, it holds that
  \[\|\rho(t_n)-\rho_n\|_1\leq c_1t_n\tau,~~~0\leq n\leq N,\]
  where the constant $c_1>0$ depends on $C_1,~C_2,~T$, but is independent of $\tau$ and $n$.
\end{theorem}
\begin{proof}
Expanding $e^{(\tau-s)A}$ in a Taylor series with remainder in integral form yields
\begin{equation*}
  e^{(\tau-s)A}=I+\int_{s}^{\tau}A \, e^{(\tau-v)A}dv.
\end{equation*}
Inserting the above equation into \eqref{equ2.4}, we have
\begin{eqnarray}\label{equ4.2}
  \rho(t_{n+1})=e^{\tau A}\rho(t_n) \, e^{\tau A^{\dag}}+\sum_{k=1}^K\gamma_k
  \int_0^{\tau}L_k \, e^{sA} \, \rho(t_n) \, e^{sA^{\dag}} \, L_k^{\dag} \, ds+\sum_{j=1}^4R_{n,j},
\end{eqnarray}
where $R_{n,1}$ is defined in \eqref{equ2.4b}, and
  \begin{align*}
  &R_{n,2}=\sum_{k=1}^K\gamma_k
  \int_0^{\tau}\left(\int_s^{\tau}Ae^{(\tau-v)A}dv\right)L_ke^{sA}\rho(t_n)e^{sA^{\dag}} L_k^{\dag} \, ds,\\
  &R_{n,3}=\sum_{k=1}^K\gamma_k
  \int_0^{\tau}L_ke^{sA}\rho(t_n)e^{sA^{\dag}} L_k^{\dag}\left(\int_s^{\tau}A^{\dag}e^{(\tau-v)A^{\dag}}dv\right)ds,\\
  &R_{n,4}=\sum_{k=1}^K\gamma_k
  \int_0^{\tau}\left(\int_s^{\tau}Ae^{(\tau-v)A}dv\right)L_ke^{sA}\rho(t_n)e^{sA^{\dag}} L_k^{\dag}\left(\int_s^{\tau}A^{\dag}e^{(\tau-v)A^{\dag}}dv\right)ds.
\end{align*}
Denote the global error $E_n:=\rho(t_n)-\rho_n$. Subtracting \eqref{equ2.5} from \eqref{equ4.2}, we obtain the error recursion
\begin{equation}\label{equ4.3}
  E_{n+1}=e^{\tau A}E_ne^{\tau A^{\dag}}+\sum_{k=1}^K\gamma_k
  \int_0^{\tau}L_ke^{sA}  \, E_n  \, e^{sA^{\dag}} L_k^{\dag} \, ds+R_n,
\end{equation}
where $R_n=\sum_{j=1}^4R_{n,j}$.

Using Lemma \ref{lem4.1} and noting that $\|\rho(t)\|_1=1$ for any $t\geq0$, we have
\begin{eqnarray}\label{equ4.4}
\|R_{n,1}\|_1&\leq&\sum_{k,j=1}^K\gamma_k\gamma_j\int_0^{\tau}\left\|e^{(\tau-s)A}L_k
  \left(\int_0^se^{(s-v)A}L_j\rho(t_n+v)L_j^{\dag}e^{(s-v)A^{\dag}}dv\right)
  L_k^{\dag}e^{(\tau-s)A^{\dag}}\right\|_1ds\nonumber\\
  &\leq& \sum_{k,j=1}^K\gamma_k\gamma_j\int_0^{\tau}\|L_k\|_1^2
  \int_0^s\left\|e^{(s-v)A}L_j\rho(t_n+v)L_j^{\dag}e^{(s-v)A^{\dag}}\right\|_1dv\nonumber\\
  &\leq& \sum_{k,j=1}^K\gamma_k\gamma_j\|L_k\|_1^2\|L_j\|_1^2
  \int_0^{\tau}\int_0^s\|\rho(t_n+v)\|_1dvds\nonumber\\
  &=&\frac{1}{2}\tau^2\sum_{k,j=1}^K\gamma_k\gamma_j\|L_k\|_1^2\|L_j\|_1^2=\frac{1}{2}C_1^2\tau^2.
\end{eqnarray}
Similarly, we can obtain
\begin{equation}\label{equ4.5}
  \|R_{n,2}\|_1\leq\frac{1}{2}C_1 \, C_2 \, \tau^2,~~~\|R_{n,3}\|_1\leq\frac{1}{2}C_1\, C_2 \, \tau^2,
\end{equation}
and
\begin{equation}\label{equ4.6}
  \|R_{n,4}\|_1\leq\frac{1}{3}C_1 \, C_2^2 \, \tau^3.
\end{equation}
Combining \eqref{equ4.4}-\eqref{equ4.6}, we have
\begin{equation}\label{equ4.7}
  \|R_n\|_1\leq c_1 \, \tau^2,
\end{equation}
where $c_1=\frac{1}{2}C_1^2+C_1C_2+\frac{1}{3}C_1C_2^2 \, T$.

It follows from $\eqref{equ4.3}$, $\eqref{equ4.7}$, $E_0=0$ and Lemma \ref{lem4.3} that
\begin{eqnarray*}
  \|E_{n+1}\|_1&\leq&\left\|e^{\tau A}E_ne^{\tau A^{\dag}}+\sum_{k=1}^K\gamma_k
  \int_0^{\tau}L_ke^{sA} \, E_n  \, e^{sA^{\dag}} L_k^{\dag} \, ds\right\|_1+\|R_n\|_1\\
  &\leq& \|E_n\|_1+c_1\tau^2\leq \|E_{n-1}\|_1+2c_1\tau^2\\
  &\leq& \|E_0\|_1+c_1(n+1)\tau^2=c_1 \, t_{n+1} \, \tau,
\end{eqnarray*}
which completes the proof.
\end{proof}

\section{A low-rank exponential Euler integrator}
When the Lindblad equation is large-sized, i.e. $m\gg 1$, the computational cost of the full-rank exponential Euler integrator \eqref{equ2.5} would be expensive since algebraic Lyapunov equations with dimension $m$ need to be solved and $2m$ matrix-vector products associated with $e^{\tau A}$ are required at every time step. Now, in order to reduce computational costs while controlling accuracy, we develop a low-rank variant of the FREE integrator \eqref{equ2.5}. The aim is to find matrices $Z_n\in \mathbb{C}^{m\times r_n}$ with $r_n\ll m$ such that the solution of the Lindblad equation can be well approximated as
\[
\rho(t_n)\approx Z_nZ_n^{\dag} =:  \varrho_n ,
\]
where we denote with $\varrho_n$ the numerical low-rank solution
to the Lindblad equation in order to distinguish it from $\rho_n$, the full-rank
numerical solution of the same equation.

Let $\rho_0\approx \varrho_0=Z_0Z_0^{\dag}$ with $Z_0\in \mathbb{C}^{m\times r_0}$ and $\mbox{Tr}(\varrho_0)=1$. To derive our low-rank exponential Euler scheme, we first define
\begin{equation}\label{equ5.1}
  \tilde{\varrho}_{n+1}=e^{\tau A}\varrho_ne^{\tau A^{\dag}}+\sum_{k=1}^K\gamma_k \, L_k  \, \tilde{W}_n \, L_k^{\dag},~~~n=0,\ldots,N-1,
\end{equation}
where
\begin{equation}\label{equ5.2}
  \tilde{W}_n=\int_0^{\tau}e^{sA}\varrho_ne^{sA^{\dag}}ds.
\end{equation}
Instead of computing $\tilde{W}_n$ by solving an algebraic Lyapunov equation as in \eqref{equ2.7a}, we approximate the integral in \eqref{equ5.2} by the following quadrature formula
\begin{equation}\label{equ5.3}
  \int_0^{\tau}e^{sA}\varrho_ne^{sA^{\dag}}ds\approx
  \tau e^{\tau A}\varrho_ne^{\tau A^{\dag}}=
  \tau e^{\tau A}Z_nZ_n^{\dag}e^{\tau A^{\dag}}:=\tau \, V_nV_n^{\dag},
\end{equation}
where $V_n=e^{\tau A}Z_n$.
It then follows from \eqref{equ5.1} that
\begin{eqnarray*}
  \tilde{\varrho}_{n+1}\approx e^{\tau A}Z_nZ_n^{\dag}e^{\tau A^{\dag}}+\sum_{k=1}^K\gamma_k \, \tau \, L_kV_nV_n^{\dag}L_k^{\dag}
  :=\tilde{Z}_{n+1}\tilde{Z}_{n+1}^{\dag},
\end{eqnarray*}
where
\[\tilde{Z}_{n+1}=\left[V_n,\sqrt{\gamma_1\tau}L_1V_n,\ldots,\sqrt{\gamma_K\tau}L_KV_n\right].\]
By this notation we mean that the matrix $V_n$ and the $K$ matrices $\sqrt{\gamma_k\tau}L_kV_n$ are placed side by side.

Note that for many problems, the matrix exponential $e^{\tau A}$ or the product of matrix exponential times vectors $e^{\tau A}Z_n$  may be costly to compute and approximations may be
required. In the low-rank algorithm, we denote with $\mathfrak{e}^{\tau A}$ (resp. $\mathfrak{e}^{\tau A}Z_n$) an approximation of $e^{\tau A}$ (resp. $e^{\tau A}Z_n$). We will always assume that $\|e^{\tau A}\sigma e^{\tau A^{\dag}}-\mathfrak{e}^{\tau A}\sigma \mathfrak{e}^{\tau A^{\dag}}\|_1\leq C_3\, tol_1 \, \|\sigma\|_1$ for any $\sigma\in \mathbb{C}^{m\times m}$, where $tol_1>0$ is the error tolerance of the underlying matrix exponential algorithm and $C_3>0$.
In addition, note that the matrix $\tilde{Z}_{n+1}$ has more columns than $Z_n$, and likely also more than its rank. A better low-rank approximation can be obtained by applying a column compression technique \cite{Lang}, which can be done by truncating the singular value decomposition (SVD) of the given matrix. We denote with $\mathcal{T}_{tol_2}(\cdot)$ the truncated SVD of a matrix with error tolerance $tol_2>0$, that is, $\mathcal{T}_{tol_2}(Z)$ represents the best rank $r$ approximation of the matrix $Z\in \mathbb{C}^{m\times q}$ in Frobenius norm, where $r$ is the minimal integer such that $\sum_{j=r+1}^q\sigma^2_{j}(Z)\leq tol_2$. We then get
\begin{equation*}
  \left\|ZZ^{\dag}-\mathcal{T}_{tol_2}(Z)\mathcal{T}_{tol_2}(Z)^{\dag}\right\|_1=
  \sum_{j=r+1}^q\sigma^2_{j}(Z)\leq tol_2.
\end{equation*}

To summarize, our low-rank exponential Euler (LREE) scheme reads as follows:
\begin{subequations}\label{equ5.4}
  \begin{align}
  &V_n=\mathfrak{e}^{\tau A}Z_n,\label{equ5.4a}\\
  &\tilde{Z}_{n+1}=\left[V_n,\sqrt{\gamma_1\tau}L_1V_n,\ldots,\sqrt{\gamma_K\tau}L_KV_n\right],\label{equ5.4b}\\
  &\hat{Z}_{n+1}= \mathcal{T}_{tol_2}(\tilde{Z}_{n+1}),\label{equ5.4c}\\
  & Z_{n+1} = \frac{\hat{Z}_{n+1}}{\|\hat{Z}_{n+1}\|_F}.\label{equ5.4d}
\end{align}
\end{subequations}
Note that the LREE integrator requires to approximate $r_n$ matrix-exponential-vector products and
 compute a SVD of a tall matrix with $(K+1)r_n$ columns at every time step. The computational cost is much cheaper
than that of the FREE scheme, which requires to compute $2m$ $(2m\gg r_n)$ matrix-exponential-vector products and solve an algebraic Lyapunov equation of dimension $m$.
\begin{rem}
Note from $\varrho_{n+1}=Z_{n+1}Z_{n+1}^{\dag}$ and \eqref{equ5.4d} that the LREE scheme \eqref{equ5.4} is positivity and trace preserving:
\begin{equation}\label{equ5.5}
  \|\varrho_{n+1}\|_1=\|Z_{n+1}Z_{n+1}^{\dag}\|_1=\mbox{Tr}(Z_{n+1}Z_{n+1}^{\dag})=
  \frac{\mbox{Tr}(\hat{Z}_{n+1}\hat{Z}_{n+1}^{\dag})}{\|\hat{Z}_{n+1}\|_F^2}=1,~~~n=0,\ldots,N-1.
\end{equation}
However, in this case normalization is performed.
\end{rem}

We remark that the LREE \eqref{equ5.4} is equivalent to
\begin{subequations}\label{equ5.6}
  \begin{align}
  &\tilde{\varrho}_{n+1}=e^{\tau A}\varrho_n \, e^{\tau A^{\dag}}+\sum_{k=1}^K\gamma_k\int_0^{\tau}L_k \, e^{sA}\varrho_n \, e^{sA^{\dag}}L_k^{\dag} \, ds,\label{equ5.6a}\\
  &\hat{\varrho}_{n+1}=\tilde{\varrho}_{n+1}-\vartheta_{n+1},\label{equ5.6b}\\
  & \varrho_{n+1}=\frac{\hat{\varrho}_{n+1}}{\mbox{Tr}(\hat{\varrho}_{n+1})},\label{equ5.6c}
\end{align}
\end{subequations}
where $\hat{\varrho}_{n+1}=\hat{Z}_{n+1}\hat{Z}_{n+1}^{\dag}$ and the Hermitian matrix $\vartheta_{n+1}$ can be seen as the perturbation caused by the numerical quadrature \eqref{equ5.3}, the approximation to the matrix exponential times vector \eqref{equ5.4a} and the column compression procedure \eqref{equ5.4c}.
A bound on $\vartheta_{n+1}$ is given in the following theorem.

\begin{theorem}\label{thm5.0}
Assume that the matrix exponential algorithm used in the LREE scheme \eqref{equ5.4} satisfies
 $\|e^{\tau A}\sigma e^{\tau A^{\dag}}-\mathfrak{e}^{\tau A}\sigma\mathfrak{e}^{\tau A^{\dag}}\|_1\leq C_3 \, tol_1\, \|\sigma\|_1$ for any $\sigma\in \mathbb{C}^{m\times m}$ with tolerance $tol_1>0$. Let $tol_2>0$ be the error tolerance of the column compression algorithm used in the LREE scheme \eqref{equ5.4}.
  Then, it holds that
  \begin{equation}\label{equ5.7}
  \|\vartheta_{n+1}\|_1\leq C_1C_2 \, \tau^2+c_2\, tol_1+tol_2,~~~n=0,1,\ldots,N-1,
\end{equation}
where the constant $c_2>0$ depends on $C_1,~C_3,~T$, but is independent of $\tau$, $n$, $tol_1$ and $tol_2$.
\end{theorem}
\begin{proof}
  Let us first define
  \begin{equation*}
    \phi_{n+1}=e^{\tau A} \varrho_n \, e^{\tau A^{\dag}}+\tau\sum_{k=1}^K\gamma_k \, L_k e^{\tau A}  \,  \varrho_n \, e^{\tau A^{\dag}} \, L_k^{\dag},
  \end{equation*}
  and
  \begin{equation*}
    \varphi_{n+1}=\mathfrak{e}^{\tau A} \varrho_n \, \mathfrak{e}^{\tau A^{\dag}}+\tau\sum_{k=1}^K\gamma_k \, L_k \mathfrak{e}^{\tau A} \, \varrho_n \, \mathfrak{e}^{\tau A^{\dag}} \, L_k^{\dag}.
  \end{equation*}
  Using \eqref{equ5.6b}, we see that
  \begin{eqnarray}\label{equ5.7a}
    \|\vartheta_{n+1}\|_1&=&\|\tilde{\varrho}_{n+1}-\phi_{n+1}+\phi_{n+1}-\varphi_{n+1}+
    \varphi_{n+1}-\hat{\varrho}_{n+1}\|_1 \nonumber \\
    &\leq&\|\tilde{\varrho}_{n+1}-\phi_{n+1}\|_1
    +\|\phi_{n+1}-\varphi_{n+1}\|_1+\|\varphi_{n+1}-\hat{\varrho}_{n+1}\|_1.
  \end{eqnarray}
  Note that the above three components in the perturbation $\vartheta_{n+1}$ correspond to
  errors induced by the numerical quadrature, the approximation to the matrix exponential and the column compression procedure, respectively.

  By comparing \eqref{equ5.4} with \eqref{equ5.6}, we obtain $\varphi_{n+1}=\tilde{Z}_{n+1}\tilde{Z}_{n+1}^{\dag}$, $\hat{\varrho}_{n+1}=\hat{Z}_{n+1}\hat{Z}_{n+1}^{\dag}$ and
  $\hat{Z}_{n+1}= \mathcal{T}_{tol_2}(\tilde{Z}_{n+1})$. Then, it follows that
  \begin{equation}\label{equ5.7b}
    \|\varphi_{n+1}-\hat{\varrho}_{n+1}\|_1=
    \left\|\tilde{Z}_{n+1}\tilde{Z}_{n+1}^{\dag}-\mathcal{T}_{tol_2}(\tilde{Z}_{n+1})
    \mathcal{T}_{tol_2}(\tilde{Z}_{n+1})^{\dag}\right\|_1\leq tol_2.
  \end{equation}

To estimate the second component in \eqref{equ5.7a}, we define $\upsilon_n=e^{\tau A}\varrho_ne^{\tau A^{\dag}}-\mathfrak{e}^{\tau A}\varrho_n\mathfrak{e}^{\tau A^{\dag}}$.
Using \eqref{equ5.5} and the assumption on the matrix exponential algorithm, we get
\[\|\upsilon_n\|_1\leq C_3\, tol_1\, \|\varrho_n\|_1=C_3\,  tol_1.\]
Then, we obtain
\begin{eqnarray}\label{equ5.7c}
  \|\phi_{n+1}-\varphi_{n+1}\|_1&=& \left\|\upsilon_n+\tau\sum_{k=1}^K\gamma_kL_k\upsilon_nL_k^{\dag}\right\|_1
  \leq \left(1+\tau\sum_{k=1}^K\gamma_k\|L_k\|_1^2\right)\|\upsilon_n\|_1\nonumber\\
  &\leq&(1+\tau C_1) \, C_3\, tol_1\leq c_2\, tol_1,
\end{eqnarray}
where $c_2=(1+TC_1) \, C_3$.

Now, we estimate the first component in \eqref{equ5.7a}.
Expanding $e^{sA}\varrho_ne^{sA^{\dag}}$ in Taylor series with remainder in integral form yields
\begin{equation*}
  e^{sA}\varrho_ne^{sA^{\dag}}=e^{\tau A}\varrho_ne^{\tau A^{\dag}}
  -\int_s^{\tau}e^{\mu A} \, (A\varrho_n+\varrho_n A^{\dag}) \, e^{\mu A^{\dag}} \, d\mu.
\end{equation*}
Defining
\[\omega_n=\int_0^{\tau}e^{sA}\varrho_ne^{sA^{\dag}}ds-\tau e^{\tau A}\varrho_ne^{\tau A^{\dag}},\]
and using \eqref{equ5.5} and Lemma \eqref{equ4.1}, we obtain
\begin{eqnarray*}
  \|\omega_n\|_1=\left\|\int_0^{\tau}\int_s^{\tau}e^{\mu A}(A\varrho_n+\varrho_n A^{\dag}) \, e^{\mu A^{\dag}}d\mu \, ds\right\|_1\leq \int_0^{\tau}\int_s^{\tau} 2\|A\|_1
  \, d\mu \, ds=C_2 \, \tau^2.
\end{eqnarray*}
By the definitions of $\tilde{\varrho}_{n+1}$ and $\phi_{n+1}$, we get
\begin{eqnarray}\label{equ5.7d}
  \|\tilde{\varrho}_{n+1}-\phi_{n+1}\|_1&=& \left\|\sum_{k=1}^K \, \gamma_k \, L_k\, \omega_n \, L_k^{\dag}\right\|_1
  \leq \sum_{k=1}^K\gamma_k \, \|L_k\|_1^2 \|\omega_n\|_1\leq C_1\, C_2 \, \tau^2.
\end{eqnarray}
Finally, we conclude using \eqref{equ5.7a}-\eqref{equ5.7d}.
\end{proof}

Now, we derive the error estimate of the LREE integrator \eqref{equ5.4} (or \eqref{equ5.6}).
We assume that the initial low-rank approximation satisfies
\[\|\rho_0-\varrho_0\|_1\leq\delta,\]
for some $\delta>0$.

\begin{theorem}\label{thm5.1}
  Let $\rho(t)$ be the solution of the Lindblad equation \eqref{equ1.1} with $\rho(0)=\rho_0$, and $\{\varrho_n\}_{n=0}^N$ be the corresponding numerical solution generated by the LREE scheme \eqref{equ5.4}. Assume that the matrix exponential algorithm used in the LREE scheme \eqref{equ5.4} satisfies
 $\|e^{\tau A}\sigma e^{\tau A^{\dag}}-\mathfrak{e}^{\tau A}\sigma\mathfrak{e}^{\tau A^{\dag}}\|_1\leq C_3\, tol_1\, \|\sigma\|_1$ for any $\sigma\in \mathbb{C}^{m\times m}$. Let the error tolerances of  the matrix exponential algorithm  and the column compression satisfy $tol_1=\tau\epsilon_1$ and $tol_2=\tau\epsilon_2$ for some $\epsilon_1,~\epsilon_2>0$, respectively. Then, it holds that
  \[\|\rho(t_n)-\varrho_n\|_1\leq \tilde{c}_1 \, t_n \, \tau+\delta+2c_2 \, t_n\epsilon_1
  +2 \, t_n\epsilon_2,~~~0\leq n\leq N,\]
  where $\tilde{c}_1=c_1+2C_1C_2$ and the positive constants $c_1$ and $c_2$ are as defined in Theorem \ref{thm4.1} and Theorem \ref{thm5.0}, respectively.
\end{theorem}
\begin{proof}
 First, we split the global error $\rho(t_{n+1})-\varrho_{n+1}$ as follows:
\begin{equation}\label{equ5.8}
  \rho(t_{n+1})-\varrho_{n+1}=(\rho(t_{n+1})-\rho_{n+1})+(\rho_{n+1}-\check{\rho}_{n+1})+(\check{\rho}_{n+1}-\varrho_{n+1}),
\end{equation}
where the auxiliary quantities $\check{\rho}_{n+1}$ are derived from the FREE integrator with low-rank initial value $\varrho_0$. That means
\begin{equation}\label{equ5.9}
  \check{\rho}_{n+1}=e^{\tau A}\check{\rho}_ne^{\tau A^{\dag}}+\sum_{k=1}^K\gamma_k
  \int_0^{\tau}L_ke^{sA}\check{\rho}_ne^{sA^{\dag}} L_k^{\dag}ds,~~~0\leq n\leq N-1,
\end{equation}
where $\check{\rho}_0=\varrho_0$. Note that the global error is decomposed into three components,
where the first component $\rho(t_{n+1})-\rho_{n+1}$ denotes the global error of the FREE integrator \eqref{equ2.5}, and we have presented the error estimate for this part in Theorem \ref{thm4.1}. The second component $\rho_{n+1}-\check{\rho}_{n+1}$ is the difference between the full-rank solutions with initial values $\rho_0$ and low-rank $\varrho_0$. The third component $\check{\rho}_{n+1}-\varrho_{n+1}$ is the difference of the solutions obtained with the FREE integrator and the LREE integrator with the same initial value $\varrho_0$.

Subtracting \eqref{equ5.9} from \eqref{equ2.5} and applying Lemma \ref{lem4.3}, we obtain
\begin{eqnarray}\label{equ5.10}
  \|\rho_{n+1}-\check{\rho}_{n+1}\|_1&=&\left\|e^{\tau A}(\rho_n-\check{\rho}_n)e^{\tau A^{\dag}}+\sum_{k=1}^K\gamma_k
  \int_0^{\tau}L_ke^{sA}(\rho_n-\check{\rho}_n)e^{sA^{\dag}} L_k^{\dag}ds\right\|_1\nonumber\\
  &\leq&\|\rho_n-\check{\rho}_n\|_1\leq\|\rho_0-\check{\rho}_0\|_1=\|\rho_0-\varrho_0\|_1=\delta.
\end{eqnarray}

The third error component $\check{\rho}_{n+1}-\varrho_{n+1}$ can be decomposed as
\begin{equation}\label{equ5.11}
  \check{\rho}_{n+1}-\varrho_{n+1}=(\check{\rho}_{n+1}-\tilde{\varrho}_{n+1})
  +(\tilde{\varrho}_{n+1}-\hat{\varrho}_{n+1})+(\hat{\varrho}_{n+1}-\varrho_{n+1}).
\end{equation}
Subtracting \eqref{equ5.6a} from \eqref{equ5.9} and applying Lemma \ref{lem4.3}, we have
\[\|\check{\rho}_{n+1}-\tilde{\varrho}_{n+1}\|_1\leq\|\check{\rho}_n-\varrho_n\|_1.\]
It then follows from \eqref{equ5.5}, \eqref{equ5.6}, \eqref{equ5.11}, Lemma \ref{lem3.1} and Theorem \ref{thm5.0} that
\begin{eqnarray}\label{equ5.12}
  \|\check{\rho}_{n+1}-\varrho_{n+1}\|_1&\leq& \|\check{\rho}_{n+1}-\tilde{\varrho}_{n+1}\|_1
  +\|\tilde{\varrho}_{n+1}-\hat{\varrho}_{n+1}\|_1+\|\hat{\varrho}_{n+1}-\varrho_{n+1}\|_1\nonumber\\
  &\leq& \|\check{\rho}_n-\varrho_n\|_1+\|\vartheta_{n+1}\|_1+\|(\mbox{Tr}(\hat{\varrho}_{n+1})-1)\varrho_{n+1}\|_1\nonumber\\
  &=&\|\check{\rho}_n-\varrho_n\|_1+\|\vartheta_{n+1}\|_1+|\mbox{Tr}(\tilde{\varrho}_{n+1})-\mbox{Tr}(\vartheta_{n+1})-1|\cdot\|\varrho_{n+1}\|_1\nonumber\\
  &=&\|\check{\rho}_n-\varrho_n\|_1+\|\vartheta_{n+1}\|_1+|\mbox{Tr}(\varrho_n)-\mbox{Tr}(\vartheta_{n+1})-1|\nonumber\\
  &=&\|\check{\rho}_n-\varrho_n\|_1+\|\vartheta_{n+1}\|_1+|\mbox{Tr}(\vartheta_{n+1})|\nonumber\\
  &\leq&\|\check{\rho}_n-\varrho_n\|_1+2\|\vartheta_{n+1}\|_1 \nonumber\\
  &\leq& \|\check{\rho}_n-\varrho_n\|_1+2\tau(C_1C_2\tau+c_2\epsilon_1+\epsilon_2)\nonumber\\
  &\leq&\|\check{\rho}_0-\varrho_0\|_1+2(n+1)\tau(C_1C_2\tau+c_2\epsilon_1+\epsilon_2)\nonumber\\
  &=&2t_{n+1}(C_1C_2\tau+c_2\epsilon_1+\epsilon_2).
\end{eqnarray}
The desired result follows from \eqref{equ5.8}, Theorem \ref{thm4.1}, \eqref{equ5.10} and
\eqref{equ5.12}.
\end{proof}

\section{Numerical experiments}
In this section, we present results of numerical experiments
that demonstrate the accuracy and efficiency of our FREE integrator \eqref{equ2.7} and LREE integrator \eqref{equ5.4}. All simulations are carried out using Python 3.12.4 on a laptop with Intel(R) Core(TM) i7-8565U CPU@1.80GHz and 16GB RAM. We use the codes in the Python package \emph{scipy} to compute the matrix exponential and the products of matrix exponential with vectors. The codes are based on algorithms proposed in \cite{Al-Mohy1,Al-Mohy2}. We also use the code in \emph{scipy} to solve the algebraic Lyapunov equations in the FREE scheme \eqref{equ2.7}, which is based on the Bartels-Stewart algorithm; see, e.g., \cite{Gajic,Hamilton}.

To show the accuracy of the proposed scheme, we display the relative errors defined as follows:
 \[Error=\frac{\|\rho_N-\rho(T)\|_1}{\|\rho(T)\|_1},\]
 and test its convergence order. The exact solution $\rho(T)$ is computed through the vectorized Lindblad equation \eqref{equ1.2} using matrix exponential.

The initial condition of the Lindblad equation is mostly chosen to be the Greenberger-Horne-Zeilinger (GHZ) state
\begin{eqnarray*}
  \rho(0)&=&\frac{1}{2}\left(|0\rangle^{\otimes K}\langle0|^{\otimes K}+
|0\rangle^{\otimes K}\langle d-1|^{\otimes K}+|d-1\rangle^{\otimes K}\langle0|^{\otimes K}+|d-1\rangle^{\otimes K}\langle d-1|^{\otimes K}\right)\\
&=&\frac{1}{2}\begin{bmatrix}
  1&0&\cdots&0&1\\
  0&0&\cdots&0&0\\
  \vdots&\vdots&\ddots&\vdots&\vdots\\
  0&0&\cdots&0&0\\
  1&0&\cdots&0&1
\end{bmatrix}\in \mathbb{R}^{d^K\times d^K}.
\end{eqnarray*}
  The GHZ state for $K$ qudits is a highly entangled quantum state that is a generalization of the $K$-qubit GHZ state, which is a superposition of all qubits being in state $|0\rangle$ and all qubits being in state $|1\rangle$. GHZ states are a valuable resource in quantum information processing due to their strong entanglement properties, which can be utilized for various tasks such as quantum teleportation, quantum error correction, and quantum cryptography.

We consider the Lindblad equation with the X-X Ising Chain Hamiltonian given by
\begin{equation}\label{equ6.1}
  H(t)=\sum_{k=1}^K(aJ_z^{(k)}+b(J_z^{(k)})^2)+\sum_{k=1}^{K-1}\sum_{l=k+1}^Kg_{kl}(t)J_x^{(k)}J_x^{(l)},
\end{equation}
where
\[J_w^{(k)}=\underbrace{I_d\otimes\cdots\otimes I_d}_{k-1}\otimes J_w\otimes \underbrace{I_d\otimes\cdots\otimes I_d}_{K-k},~~~w=x,z,~~~k=1,\ldots,K,\]
$I_d$ is the $d\times d$ identity matrix and the $J_x,J_z$ operators are defined as follows
 \begin{equation*}
J_x=\frac{1}{2}\begin{bmatrix}
  0&\sqrt{d-1}&0&\cdots&0\\
  \sqrt{d-1}&0&\sqrt{2(d-2)}&0&\vdots\\
  \vdots&\ddots&\ddots&\ddots&\vdots\\
  \vdots&\ddots&\sqrt{(d-2)2}&0&\sqrt{d-1}\\
  0&\cdots&0&\sqrt{d-1}&0
\end{bmatrix}\in \mathbb{R}^{d\times d},
 \end{equation*}
 and
 \begin{equation*}
   J_z=\begin{bmatrix}
  \frac{d-1}{2}&0&\cdots&0\\
  0&\frac{d-3}{2}&0&\vdots\\
  \vdots&\ddots&\ddots&0\\
  0&\cdots&0&\frac{1-d}{2}
\end{bmatrix}\in \mathbb{R}^{d\times d}.
 \end{equation*}
The first sum in \eqref{equ6.1} is the so-called drift-Hamiltonian, which characterizes the inner dynamics of the qudit. The term $aJ_z^{(k)}$ allows for an even spacing of eigenvalues of amplitude $a$ between the levels, while the term $b(J_z^{(k)})^2$ allows for an uneven spacing of amplitude $b$ and therefore, the $d$ levels become individually addressable. The second sum is the interaction Hamiltonians, where $g_{kl}(t)$ is the coupling strength. If $g_{kl}$ depends on time, we have a
time-dependent Hamiltonian. The Ising-type Hamiltonian describes the interaction between two qudits as a product of their respective local operators.
The jump operators $L_k$ are always set to be $L_k=J_z^{(k)}$, unless stated otherwise.

\subsection{Tests on the FREE integrator}

\begin{figure}
{ \centering
\includegraphics[width=0.48\textwidth,height=0.23\textheight]{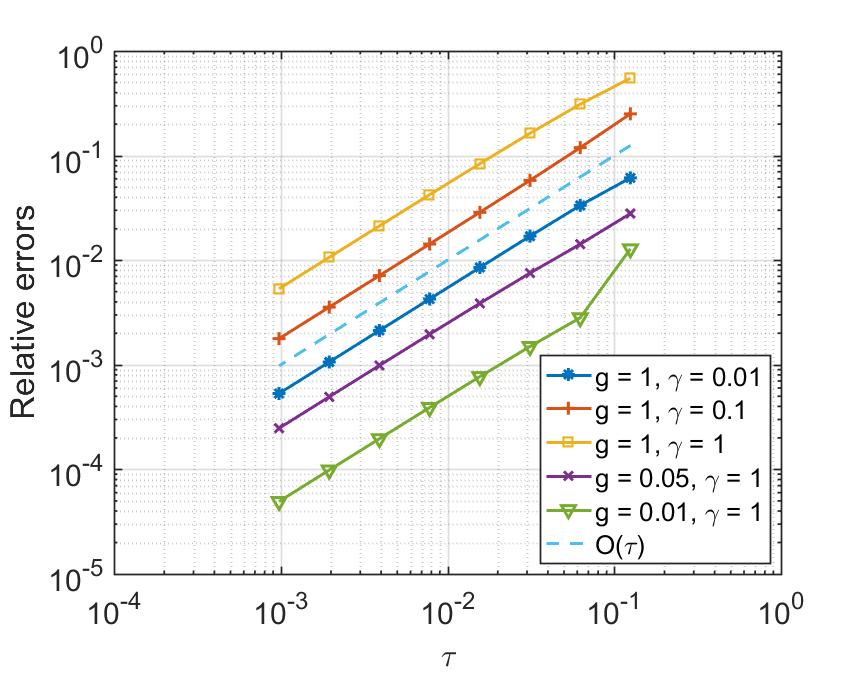}
} { \centering
\includegraphics[width=0.48\textwidth,height=0.23\textheight]{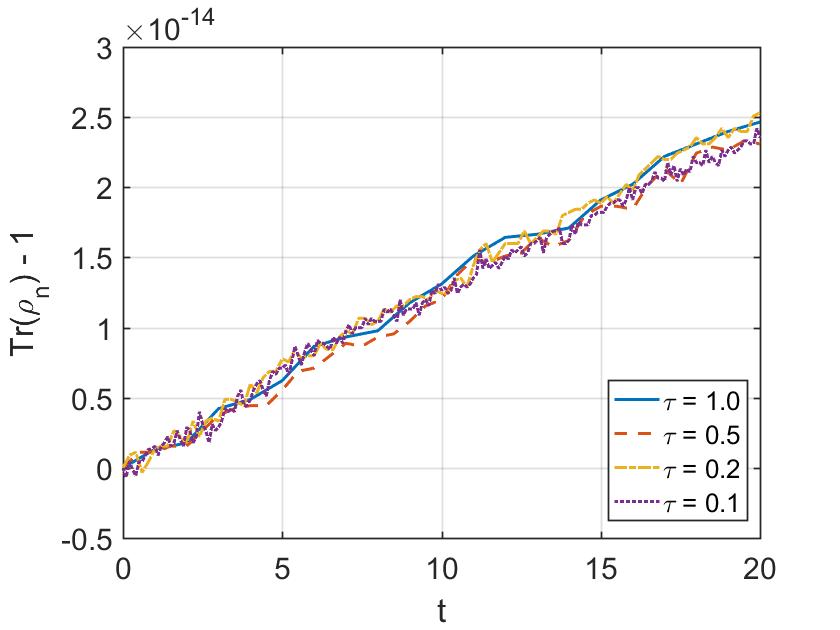}
}\\ { \centering
\includegraphics[width=0.48\textwidth,height=0.23\textheight]{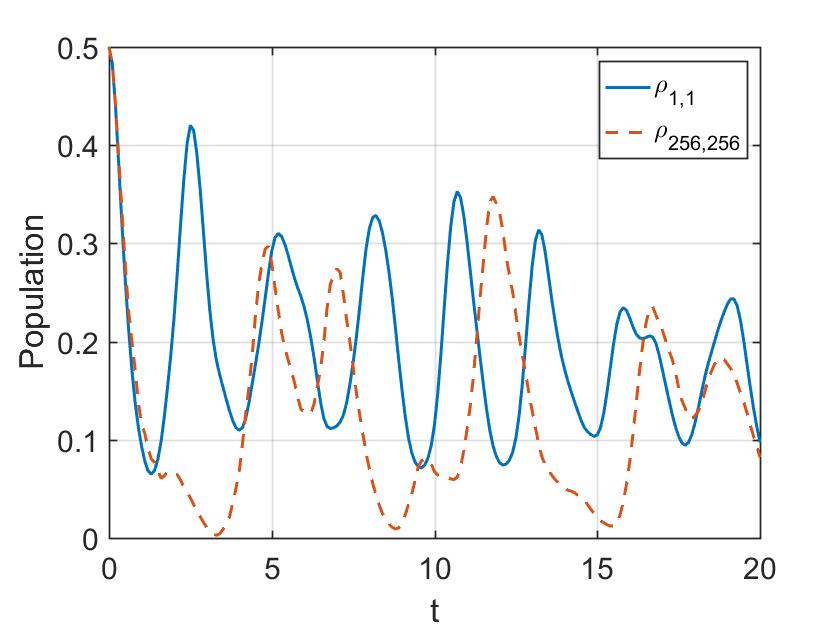}
} { \centering
\includegraphics[width=0.48\textwidth,height=0.23\textheight]{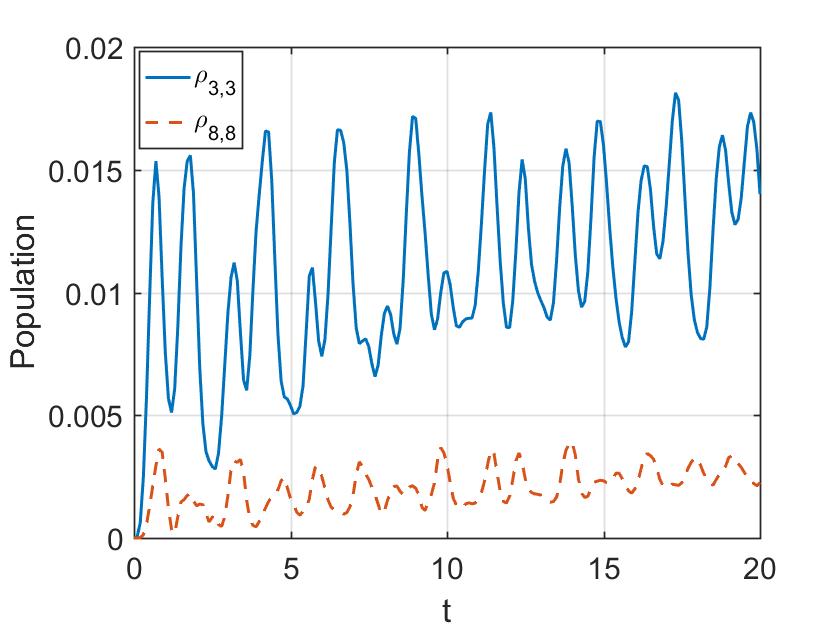}
}\caption{Numerical results of the FREE integrator for the Lindblad equation with Hamiltonian \eqref{equ6.1} ($d=4$, $K=4$, $a=1.5$, $b=0.5$, $\gamma_k=\gamma$, $g_{kl}(t)\equiv g$). Top left: relative errors at $T=1$ vs step sizes for different $g$ and $\gamma$. Top right: evolutions of $\mbox{Tr}(\rho_n)-1$ vs time for different $\tau$ with fixed $g=1$, $\gamma=0.01$ and $T=20$. Bottom left: evolutions of the populations $\rho_{1,1}$ and $\rho_{256,256}$ with $\tau=0.1$, $g=1$, $\gamma=0.01$ and $T=20$. Bottom right: evolutions of the populations $\rho_{3,3}$ and $\rho_{8,8}$ with $\tau=0.1$, $g=1$, $\gamma=0.01$ and $T=20$.}
\label{fig6.1}
\end{figure}

In Figure \ref{fig6.1}, we present numerical results of the proposed FREE integrator for the Lindblad equation with time-independent Hamiltonian ($g_{kl} =g=1$). The solvers for matrix
exponential and Lyapunov equation used in FREE integrator were set with default tolerance (machine precision $10^{-16}$).
We see that the expected first-order accuracy is achieved for the FREE scheme, which is in agreement with the convergence estimate stated in Theorem \ref{thm4.1}. It can also be seen that $\mbox{Tr}(\rho_n)-1$ is close to machine precision for different time step sizes, which demonstrates the preservation of unit trace unconditionally. We also see that the populations $\rho_{1,1},\rho_{3,3},\rho_{8,8}$ and $\rho_{256,256}$ stay nonnegative in the interval $[0,1]$, which is consistent with the positive property of the density matrix. Note that although we display evolutions of populations $\rho_{1,1},\rho_{3,3},\rho_{8,8}$ and $\rho_{256,256}$, the positive features are visible in all populations.

In Figure \ref{fig6.2}, we display numerical results of the FREE scheme for the Lindblad equation with time-dependent Hamiltonian. Here the coupling strength was set to
 \[g_{kl}(t)=\delta_{k,l-1}\cdot g(t)=\left\{\begin{array}{ll}
g(t)~~~\mbox{if}~~~k=l-1,\\
0~~~~~~\mbox{else},
\end{array} \right.\]
and we chose $g(t)$ as $(1+t)^{\frac{1}{4}}$, $e^{-t}$, $\sin(2\pi t)$ and $10\sin(10\pi t)$, respectively.
 Again, we see that the FREE scheme achieves first-order accuracy in time. The numerical solutions obtained by our FREE scheme preserve positivity and unit trace unconditionally.

\begin{figure}
{ \centering
\includegraphics[width=0.48\textwidth,height=0.23\textheight]{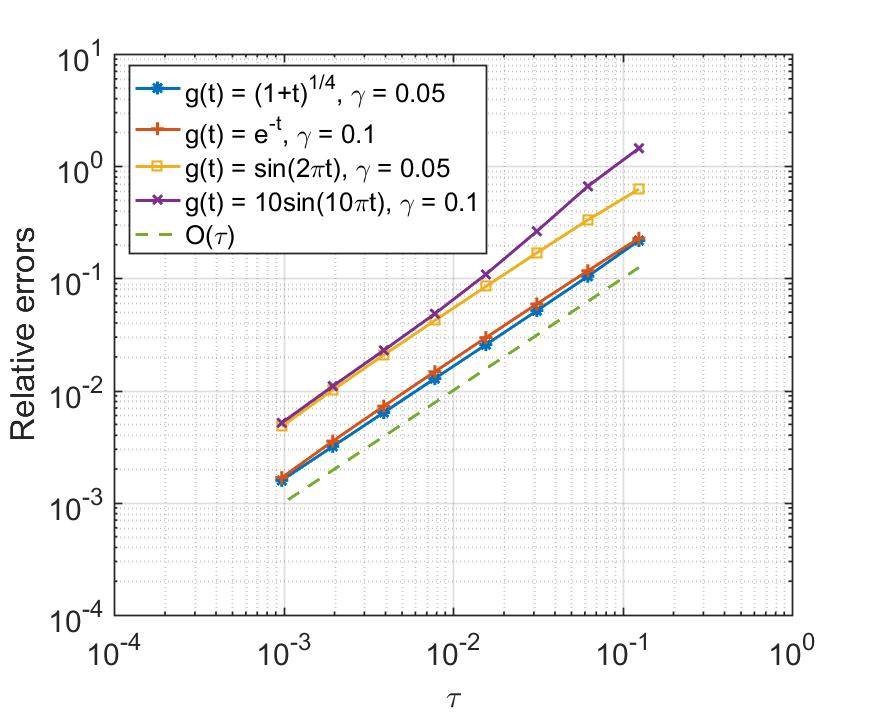}
} { \centering
\includegraphics[width=0.48\textwidth,height=0.23\textheight]{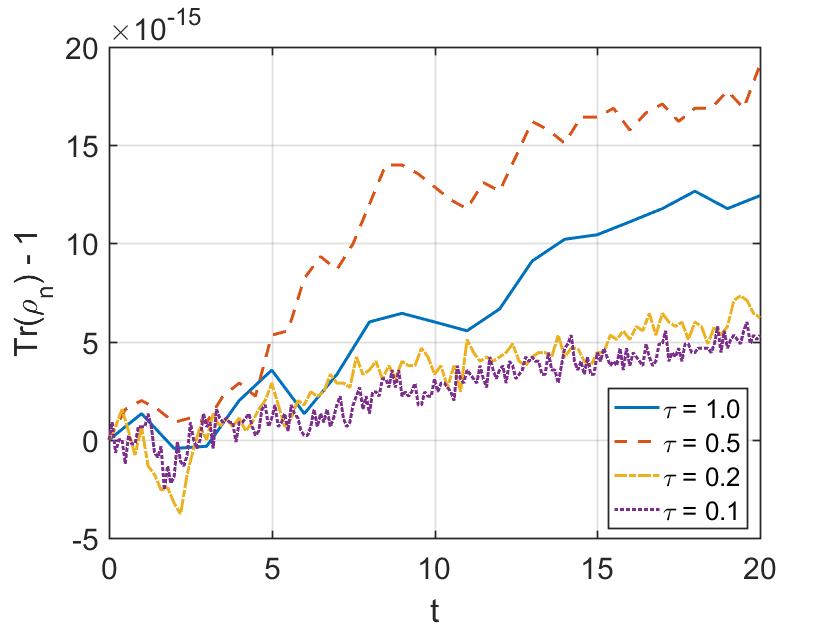}
}\\ { \centering
\includegraphics[width=0.48\textwidth,height=0.23\textheight]{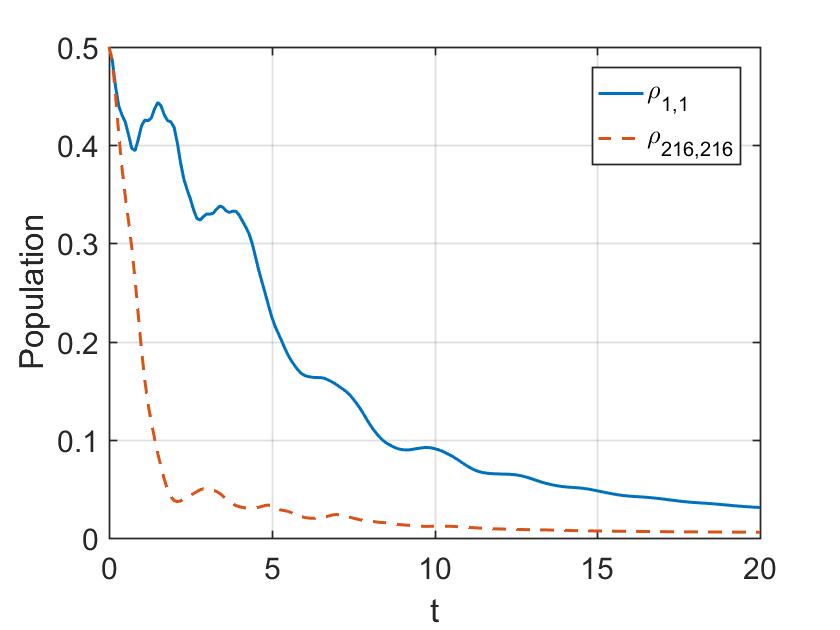}
} { \centering
\includegraphics[width=0.48\textwidth,height=0.23\textheight]{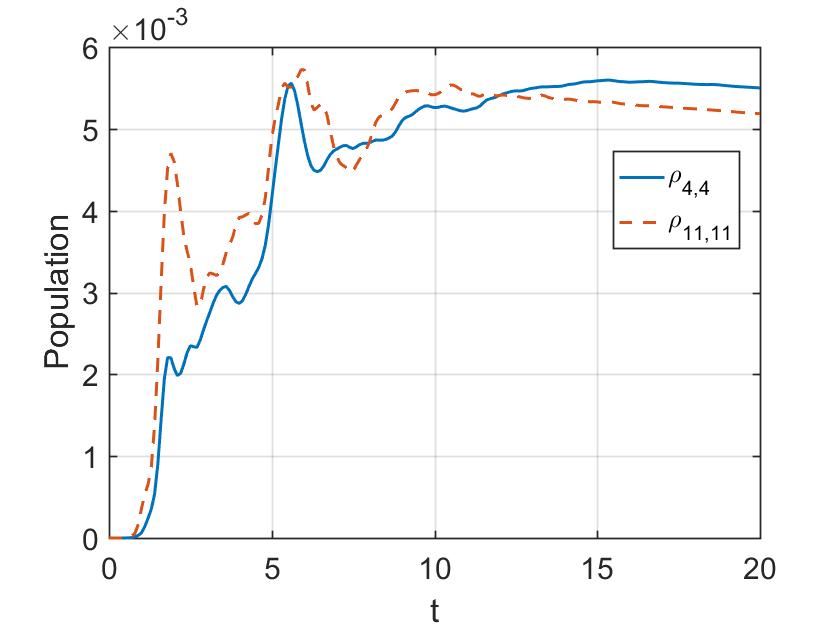}
}\caption{Numerical results of the FREE integrator for the Lindblad equation with time-dependent Hamiltonian \eqref{equ6.1} ($d=6$, $K=3$, $a=1$, $b=1$, $\gamma_k=\gamma$, $g_{kl}(t)= \delta_{k,l-1}\cdot g(t)$). Top left: relative errors at $T=1$ vs step sizes for different $g(t)$ and $\gamma$. Top right: evolutions of $\mbox{Tr}(\rho_n)-1$ vs time for different $\tau$ with $g(t)=(1+t)^{\frac{1}{4}}$, $\gamma=0.05$ and $T=20$. Bottom left: evolutions of the populations $\rho_{1,1}$ and $\rho_{216,216}$ with $\tau=0.1$, $g(t)=(1+t)^{\frac{1}{4}}$, $\gamma=0.05$ and $T=20$. Bottom right: evolutions of the populations $\rho_{4,4}$ and $\rho_{11,11}$ with $\tau=0.1$, $g(t)=(1+t)^{\frac{1}{4}}$, $\gamma=0.05$ and $T=20$.}
\label{fig6.2}
\end{figure}

\subsection{Tests on the LREE integrator}
Now, we test the performance of the LREE scheme for the Lindblad equation with Hamiltonian \eqref{equ6.1}. First, we test the error behavior of the LREE scheme with respect to time discretization and the initial low-rank error $\delta$. In this case, the initial condition of the Lindblad equation
is set to be
\[\rho(0)=\left(1-\frac{\delta}{2}\right)q_1q_1^\top+\frac{\delta}{2}q_2q_2^\top,\]
where $q_1$ and $q_2$ are orthonormal vectors and we obtain them from SVD
of a random $m\times 2$ real matrix. The low-rank initial factor is set to $Z_0=q_1$ and clearly we have $\|\rho(0)-Z_0Z_0^{\dag}\|_1=\delta$. The error tolerances of the matrix exponential and column compression are set to be $tol_1=10^{-10}$ and $tol_2=10^{-10}$, respectively. In Figure \ref{fig6.3}, we present the error behavior of the LREE scheme with various choices of $\delta$. As we can see, the LREE scheme has convergence of order one when the initial low-rank error is small enough. When the initial low-rank error is dominant, decreasing the step size does not lead to more accurate solution. These numerical results are in agreement with the error estimate stated in Theorem \ref{thm5.1} and we note that the estimate with respect to $\delta$ can not be made sharper.

\begin{figure}
{ \centering
\includegraphics[width=0.48\textwidth]{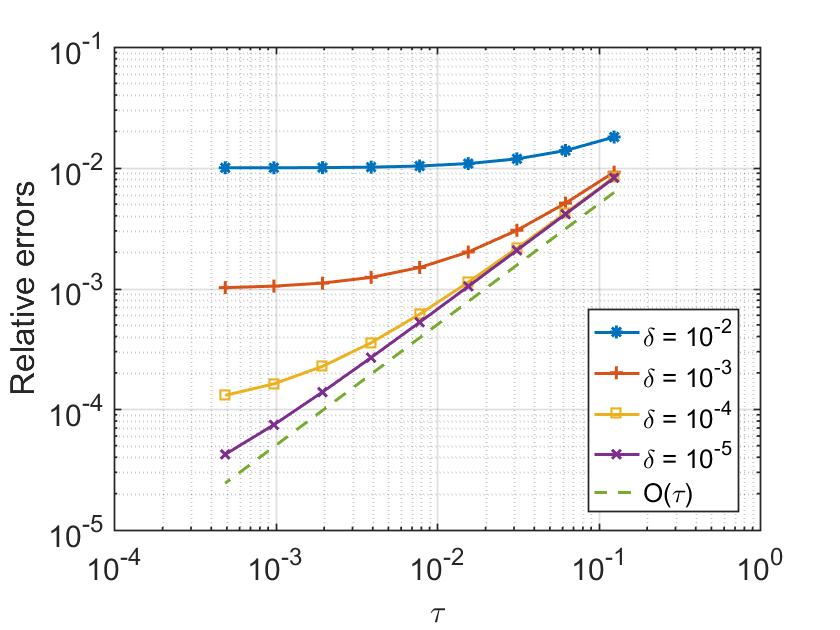}
}  { \centering
\includegraphics[width=0.48\textwidth]{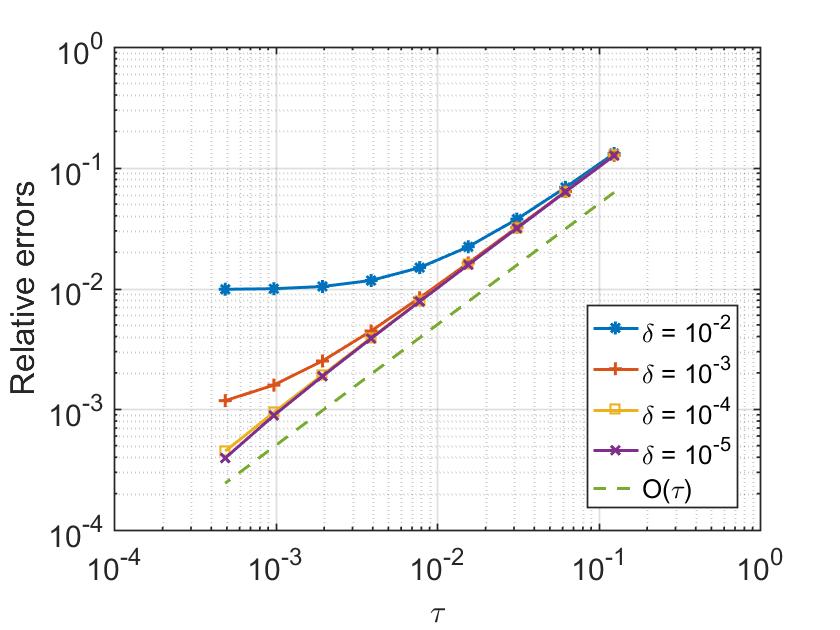}
} \caption{Numerical results of the LREE scheme with $tol_1=10^{-10}$ and $tol_2=10^{-10}$ for the Lindblad equation with Hamiltonian \eqref{equ6.1}. Left: relative errors at $T=1$ vs step sizes for the Lindblad problem with $d=4$, $K=4$, $a=1.5$, $b=0.5$, $\gamma_k=\gamma=0.01$, $g_{kl}(t)\equiv g=1$. Right: relative errors at $T=1$ vs step sizes for the Lindblad problem with $d=6$, $K=3$, $a=1$, $b=1$, $\gamma_k=\gamma=0.05$, $g_{kl}(t)= \delta_{k,l-1}\cdot(1+t)^{\frac{1}{4}}$.}
\label{fig6.3}
\end{figure}

Second, we focus on error behavior with respect to the time discretization and low-rank approximation. The initial condition is chosen to be the GHZ state. The initial low-rank factor is set to be $Z_0=\frac{1}{\sqrt{2}}[1,0,\ldots,0,1]^\top$. In this case, $\rho(0)=Z_0Z_0^{\dag}$ and there is no initial low-rank approximation error ($\delta=0$). The error tolerance of the code for matrix exponential is set to be $tol_1=10^{-10}$. The error tolerance of the column compression algorithm is set to be $tol_2=\epsilon_2\tau$. In Figure \ref{fig6.4}, we present the error behavior of the LREE scheme with various $\epsilon_2$. We see that the LREE scheme achieves first-order accuracy when the low-rank approximation error (due to the column compression) is small enough. Decreasing the step size cannot lead to more accurate solution when the low-rank approximation error is dominant. These results are also in agreement with the error estimate stated in Theorem \ref{thm5.1}.

\begin{figure}
{ \centering
\includegraphics[width=0.48\textwidth]{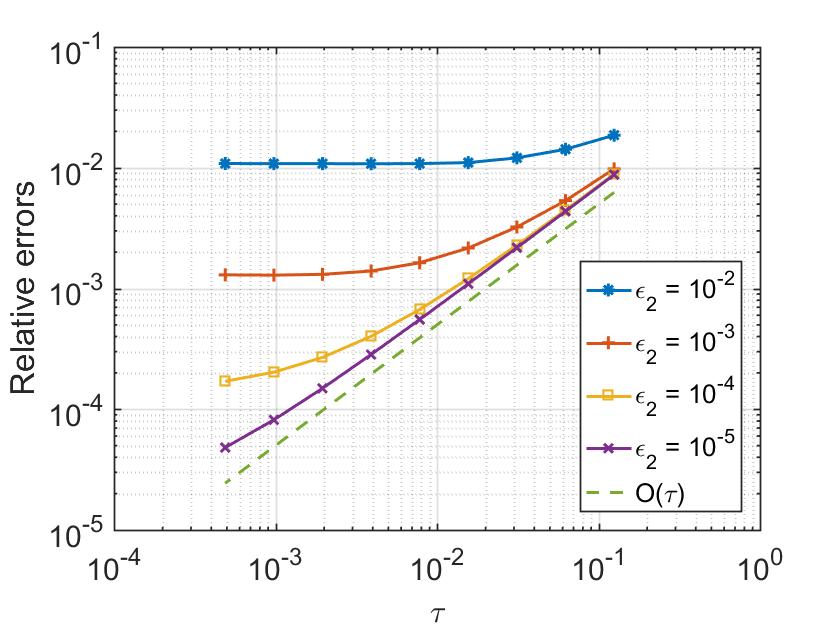}
}  { \centering
\includegraphics[width=0.48\textwidth]{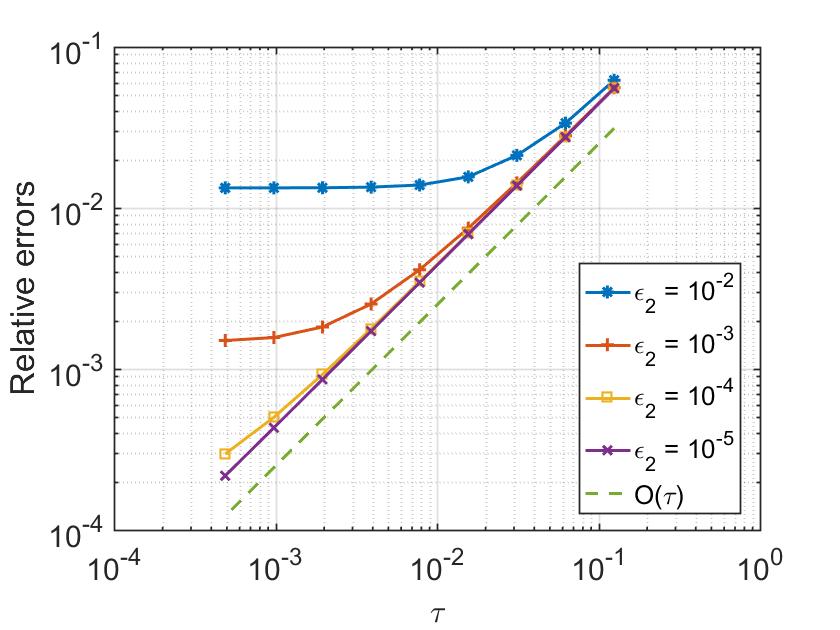}
} \caption{Numerical results of the LREE scheme  with $\delta=0$ and $tol_1=10^{-10}$ for the Lindblad equation with Hamiltonian \eqref{equ6.1}. Left: relative errors at $T=1$ vs step sizes for the Lindblad problem with $d=4$, $K=4$, $a=1.5$, $b=0.5$, $\gamma_k=\gamma=0.01$, $g_{kl}(t)\equiv g=1$. Right: relative errors at $T=1$ vs step sizes for the Lindblad problem with $d=6$, $K=3$, $a=1$, $b=1$, $\gamma_k=\gamma=0.05$, $g_{kl}(t)= \delta_{k,l-1}\cdot(1+t)^{\frac{1}{4}}$.}
\label{fig6.4}
\end{figure}

Finally, we test the performance of the LREE scheme in terms of errors due to time discretization and the approximation of the products of matrix exponential with vectors. We set
$Z_0=\frac{1}{\sqrt{2}}[1,0,\ldots,0,1]^\top$ and in this case $\delta=0$. The error tolerance of the column compression is set to $tol_2=10^{-10}$. The error tolerance of the algorithm for matrix exponential is set to be $tol_1=\epsilon_1\tau$. In Figure \ref{fig6.5}, we show the error behavior of the LREE scheme with various $\epsilon_1$. We see that the LREE scheme achieves first-order accuracy even when $\epsilon_1$ is much bigger than $\tau$. This error behavior is different from those observed in Figures \ref{fig6.3} and \ref{fig6.4}. We remark that this performance may be caused by the matrix exponential algorithm \cite{Al-Mohy2} that we used. In the Theorem \ref{thm5.1}, we assume that the matrix exponential algorithm satisfies $\|e^{\tau A}\sigma e^{\tau A^{\dag}}-\mathfrak{e}^{\tau A}\sigma\mathfrak{e}^{\tau A^{\dag}}\|_1\leq C_3\, tol_1\|\sigma\|_1$, where $tol_1$ is the input error tolerance. It seems that the error constant $C_3$ depends on $\tau$, which was verified in Figure \ref{fig6.6}, in which we show the values of $C_e:=\|e^{\tau A_n}\sigma e^{\tau A_n^{\dag}}-\mathfrak{e}^{\tau A_n}\sigma\mathfrak{e}^{\tau A_n^{\dag}}\|_1/tol_1$ for various $\tau$ and $tol_1$, where $\sigma>0$ is a random matrix with $\|\sigma\|_1=1$ and we set $A_n=A$ for equation with time-independent Hamiltonian and $A_n=A_{10}$ for problem with time-dependent Hamiltonian.

\begin{figure}
{ \centering
\includegraphics[width=0.48\textwidth]{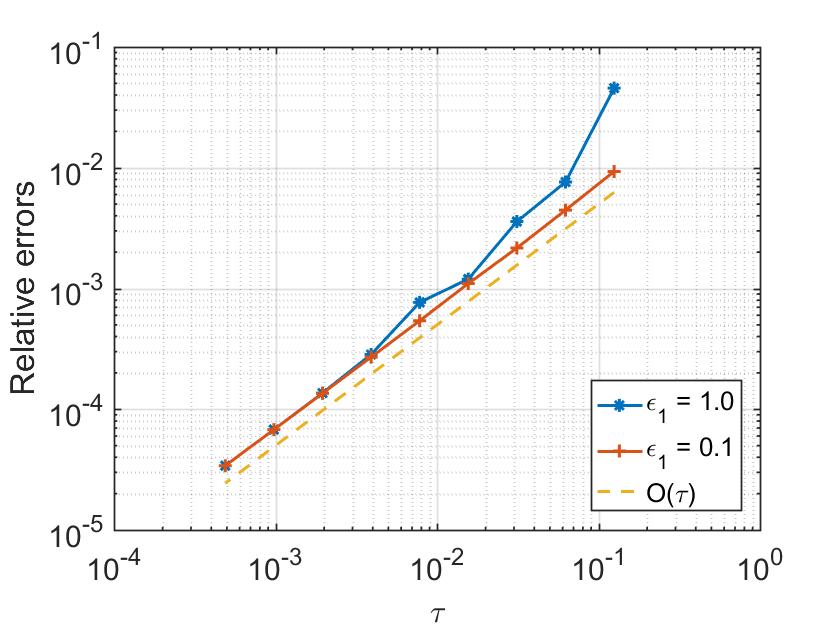}
}  { \centering
\includegraphics[width=0.48\textwidth]{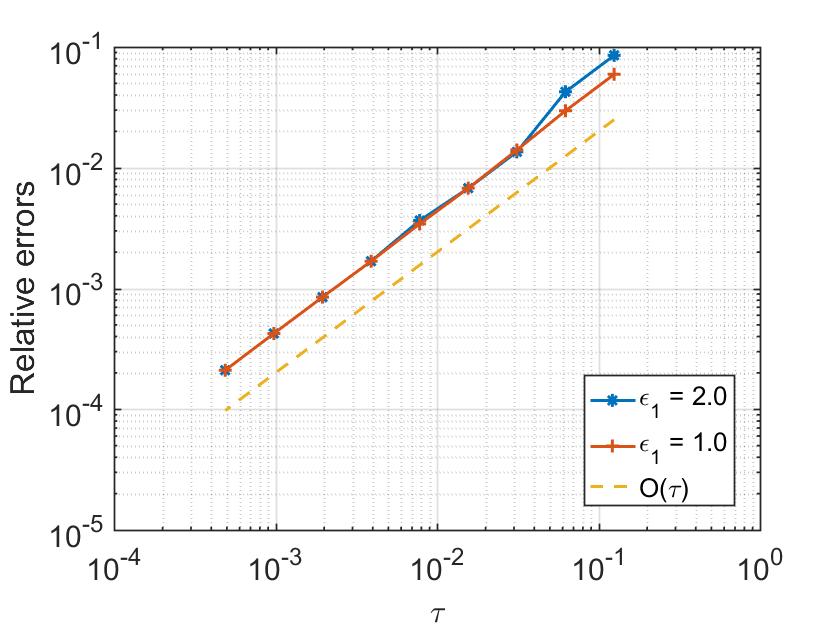}
} \caption{Numerical results of the LREE scheme with $\delta=0$ and $tol_2=10^{-10}$ for the Lindblad equation with Hamiltonian \eqref{equ6.1}. Left: relative errors at $T=1$ vs step sizes for the Lindblad problem with $d=4$, $K=4$, $a=1.5$, $b=0.5$, $\gamma_k=\gamma=0.01$, $g_{kl}(t)\equiv g=1$. Right: relative errors at $T=1$ vs step sizes for the Lindblad problem with $d=6$, $K=3$, $a=1$, $b=1$, $\gamma_k=\gamma=0.05$, $g_{kl}(t)= \delta_{k,l-1}\cdot(1+t)^{\frac{1}{4}}$.}
\label{fig6.5}
\end{figure}

\begin{figure}
{ \centering
\includegraphics[width=0.48\textwidth]{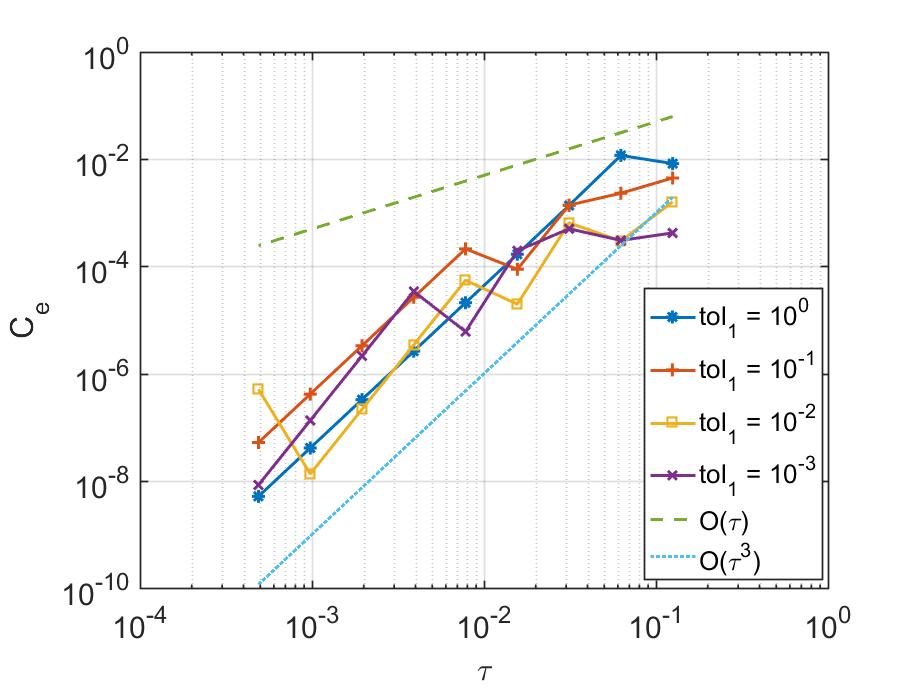}
}  { \centering
\includegraphics[width=0.48\textwidth]{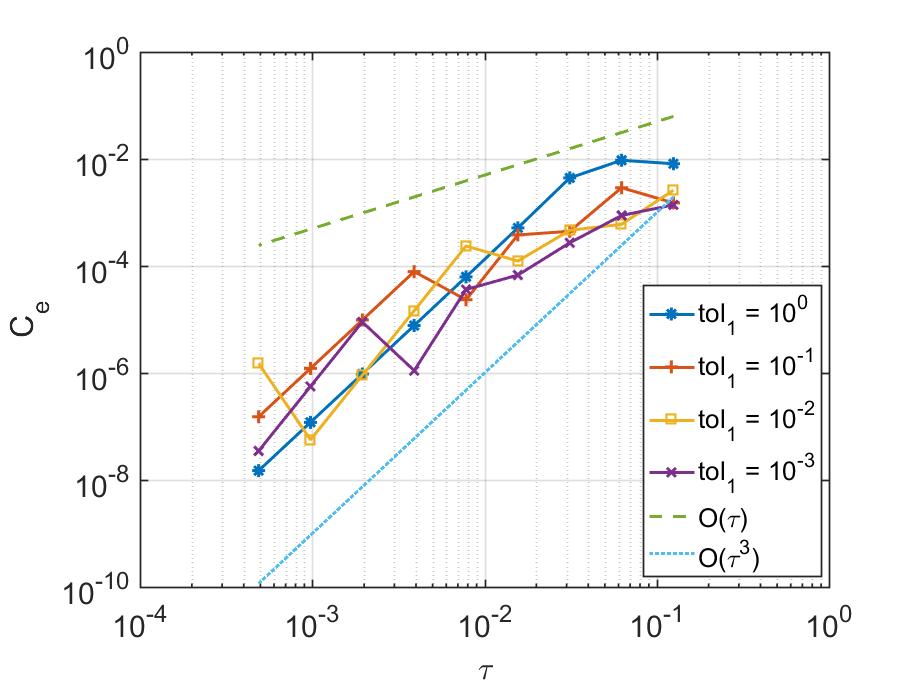}
} \caption{$C_e$ vs $\tau$. Left: time-independent case with $d=4$, $K=4$, $a=1.5$, $b=0.5$, $\gamma_k=\gamma=0.01$, $g_{kl}(t)\equiv g=1$. Right: time-dependent case with $d=6$, $K=3$, $a=1$, $b=1$, $\gamma_k=\gamma=0.05$, $g_{kl}(t)= \delta_{k,l-1}\cdot(1+t)^{\frac{1}{4}}$.}
\label{fig6.6}
\end{figure}

\FloatBarrier

\subsection{Numerical comparison}
Now, we compare integration algorithms with the Lindblad equation suite \emph{mesolve} developed in QuTip \cite{Johansson}, a widely used Python package for the simulation of open quantum systems. We consider five ODE solvers in \emph{mesolve}: \emph{adams}, \emph{bdf}, \emph{lsoda}, \emph{dop853} and \emph{vern9}, which are based on twelfth-order Adams method, fifth-order backward differentiation formula, adaptive method chosen between \emph{adams} and \emph{bdf}, Dormand and Prince's eighth-order Runge-Kutta method  and Verner's ninth-order Runge-Kutta method (see \cite{Hairer,Verner}), respectively. Note that
these ODE solvers consider the Lindblad equation in the vectorized form \eqref{equ1.2}.

In Figure \ref{fig6.7}, we compare the performance of our FREE and LREE schemes and the solvers used in QuTip with respect to accuracy and computational times measured in seconds. Note that we only report the running time of the ODE solvers. The initial condition is set to be the GHZ state ($\delta=0$ for the LREE scheme)  and relative errors are reported at $T=1$. The error tolerances of the matrix exponential algorithm and column compression method used in LREE scheme are set to be $tol_1=\frac{\tau}{20}$ and $tol_2=\frac{\tau^2}{20}$ for equation with time-independent Hamiltonian and $tol_1=\frac{\tau}{10}$ and $tol_2=\frac{\tau^2}{10}$ for equation with time-dependent Hamiltonian, respectively. We observe that the LREE scheme performs faster than the FREE scheme for a fixed level of accuracy. For a fixed level of moderate accuracy (i.e., bigger than $10^{-3}$ in the time-independent case), our LREE scheme is more efficient than the solvers in QuTip. However, for the given problems (with $m=256$), the solvers in QuTip performs faster than our exponential integrators in the high accuracy case. After all, our methods are only first-order accurate and much more steps are required to reach similar accuracy.

\begin{figure}
{ \centering
\includegraphics[width=0.48\textwidth]{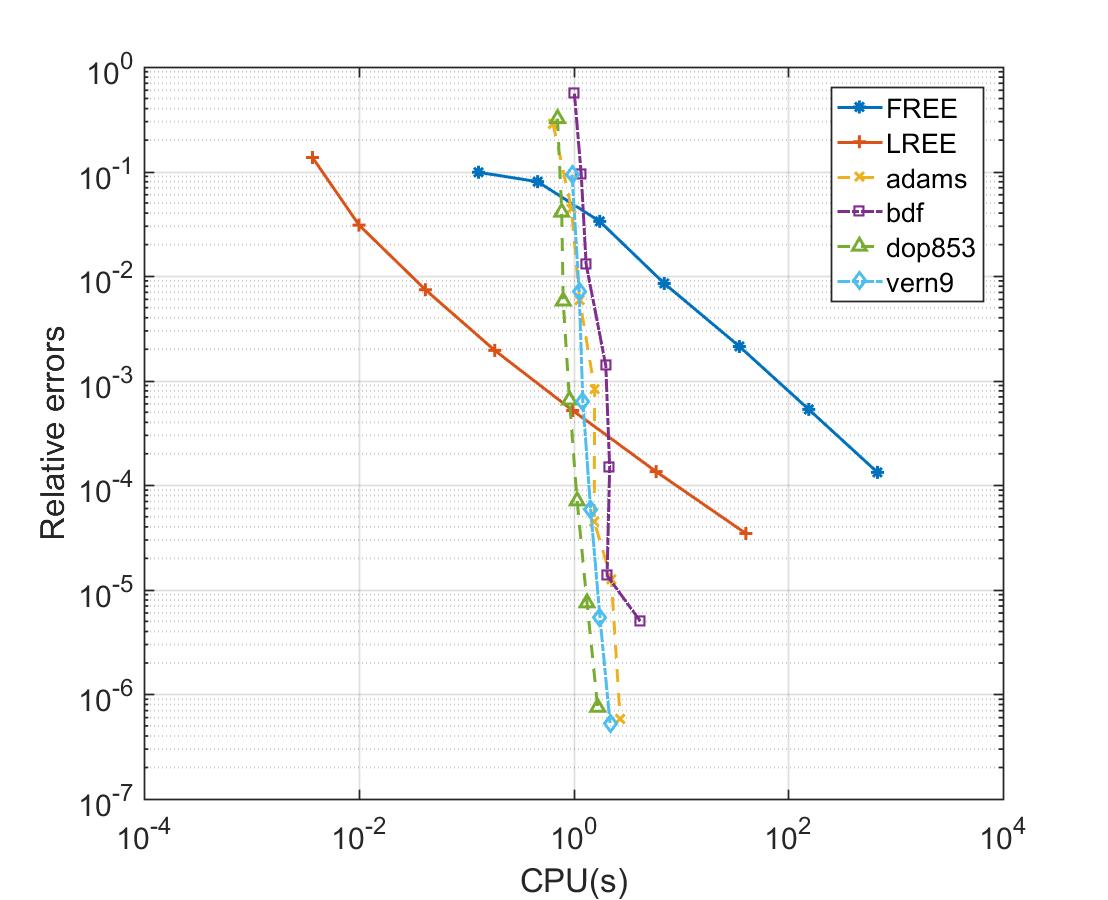}
} { \centering
\includegraphics[width=0.48\textwidth]{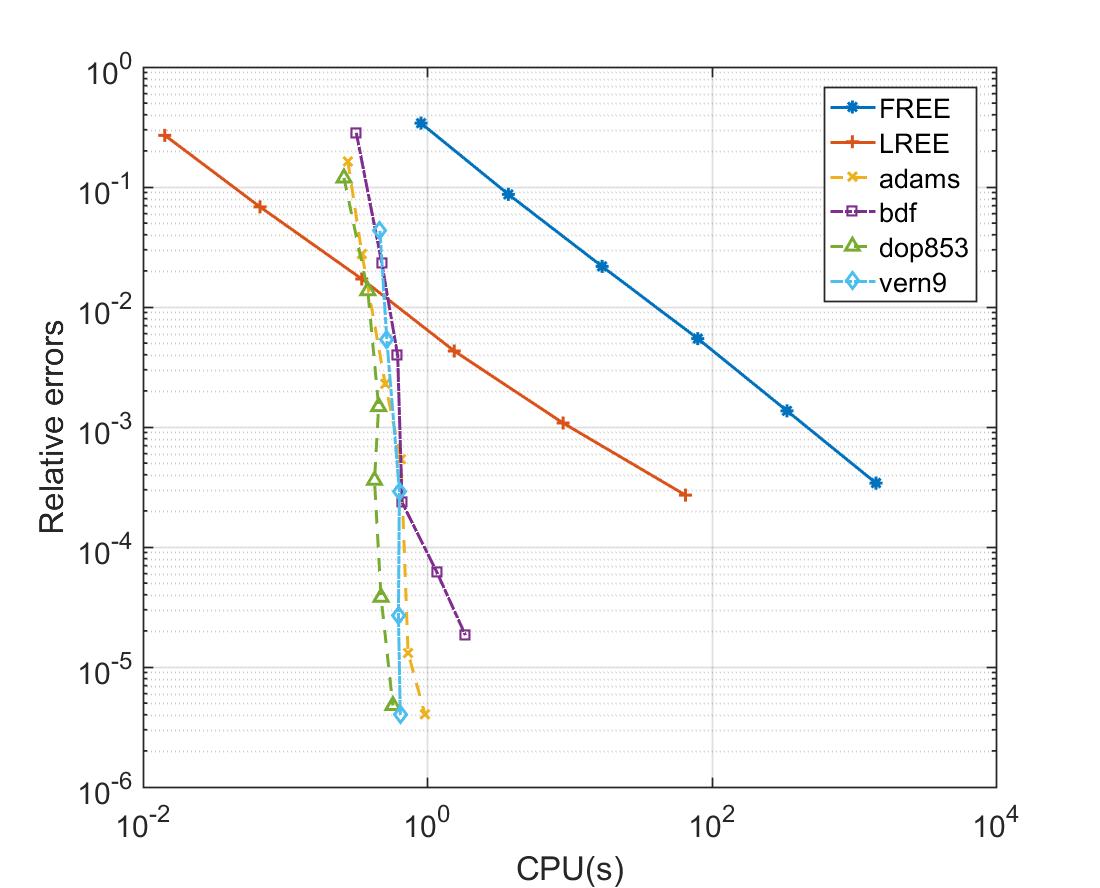}
} \caption{Numerical comparison between the proposed exponential schemes and the solvers in QuTip for the Lindblad equation with Hamiltonian \eqref{equ6.1}. Left: results for problem with $d=4$, $K=4$, $a=1.5$, $b=0.5$, $\gamma_k=\gamma=0.01$, $g_{kl}(t)\equiv g=1$. Right: results for problem with $d=4$, $K=4$, $a=1$, $b=1$, $\gamma_k=\gamma=0.05$, $g_{kl}(t)= \delta_{k,l-1}\cdot\sin(2\pi t)$.
}
\label{fig6.7}
\end{figure}

In Figure \ref{fig6.8}, we compare the computational times of our FREE and LREE schemes with those of the Lindblad solvers used in QuTip as the size $m=d^K$ of the Lindblad equation increases. We fix $K=1$ and let $d$ increase. The jump operators $L_k$ are set to be $L_k=J_x^{(k)}$. The error tolerances of the matrix exponential algorithm and column compression method used in the LREE scheme are set to be $tol_1=\frac{\tau}{10}$ and $tol_2=\frac{\tau^2}{10}$, respectively.
In this experiment, the absolute and relative tolerances of the QuTip solvers and the step sizes of the FREE and LREE schemes are chosen such that the relative errors at $T=0.1$ are approximately $10^{-3}$, see the right-hand plot in Figure \ref{fig6.8}. We see that our LREE scheme performs much faster than the QuTip Lindblad solvers for problems with high dimensions. In particular, we notice a strikingly better computational time complexity of our LREE scheme, with much smaller
change in the runtime of our algorithm with respect to increasing $m$. However, we see that the FREE scheme is less efficient than the other solvers. This is mainly due to the computation of the associated algebraic Lyapunov equations.

Note that the numerical results presented in Figures \ref{fig6.7} and \ref{fig6.8} are
with respect to problems with sparse Hamiltonian \eqref{equ6.1} and sparse jump operators. In Figure \ref{fig6.9}, we compare the computational times of our exponential schemes with those of the five Lindblad solvers used in QuTip in the case of a Lindblad equation with dense jump operator. We consider the Lindblad equation with Hamiltonian \eqref{equ6.1} and $K=1$, and the jump operator $L_1$ is chosen as a random and dense matrix. We observe that now both of our FREE and LREE schemes outperform the
five QuTip Lindblad solvers in terms of computational times. We also see that the computational times of the QuTip Lindblad solvers increase quickly as the dimension $m$ increases.

We remark that, in the dense case, the Lindblad solvers used in QuTip require to store a dense and complex $m^2 \times m^2$  matrix $\mathcal{L}$, which requires memory of $16m^4$ Bytes. On the other hand, our FREE and LREE schemes only require to store a matrix of size $m \times m$, which needs memory of $16m^2$ Bytes. For example, if $m=120$, QuTip needs nearly $3000$ MB memory while our algorithms only require $0.2$ MB memory. In the sparse case, the memory requirements of the QuTip and our exponential schemes on the coefficient matrix reduce to $\mathcal{O}(m^2)$ and $\mathcal{O}(m)$, respectively. It seems that the QuTip Lindblad \emph{lsoda} solver requires dense representation of the data matrices. So the \emph{lsoda} method was absent from the numerical comparison presented in Figures \ref{fig6.7} and \ref{fig6.8} since memory error appears for the considered dimensions.

\begin{figure}
{ \centering
\includegraphics[width=0.48\textwidth]{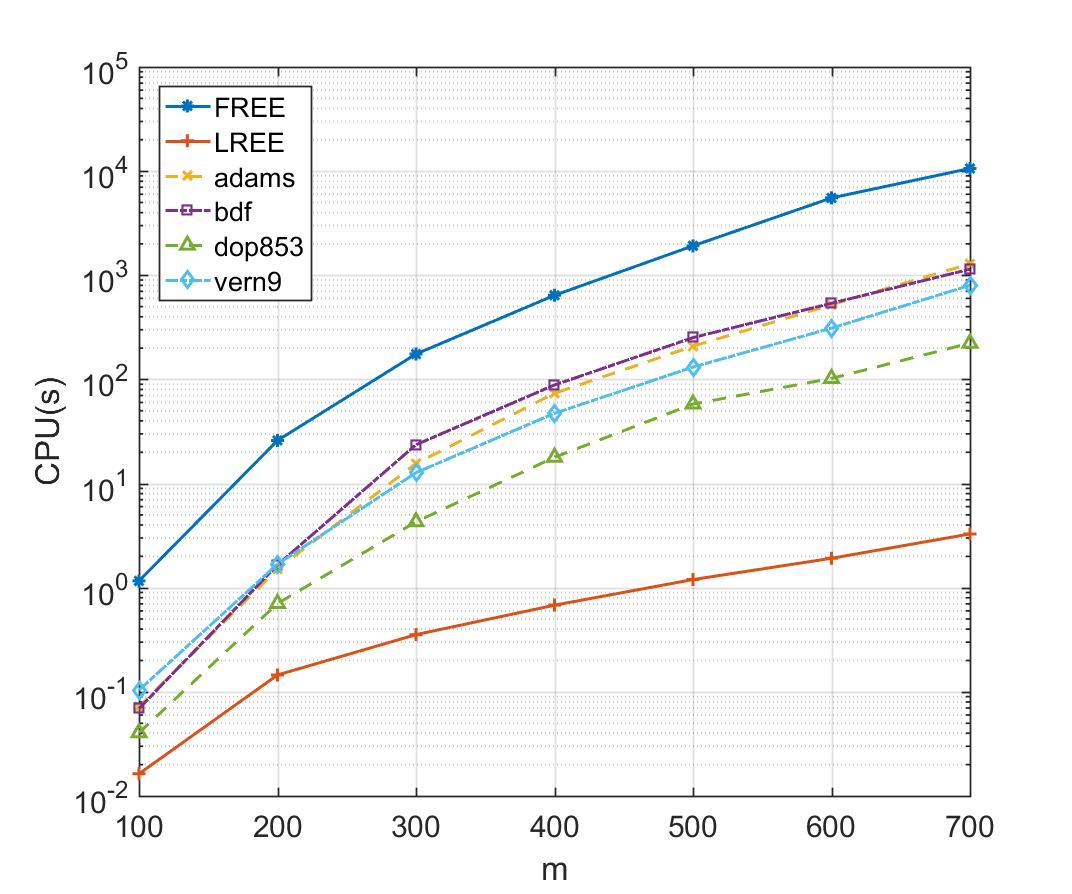}
} { \centering
\includegraphics[width=0.48\textwidth]{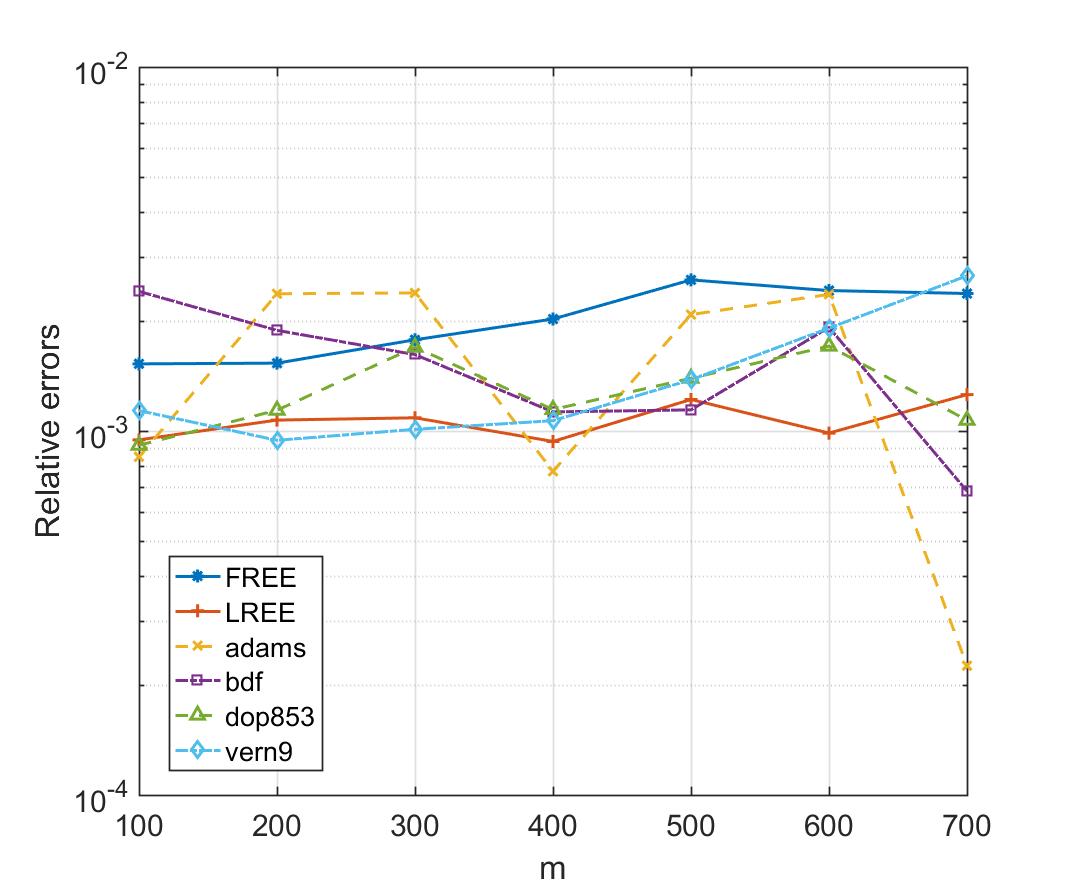}
} \caption{Numerical comparison between the proposed exponential schemes and the solvers in QuTip for the Lindblad equation with Hamiltonian \eqref{equ6.1} and jump operators $L_k=J_x^{(k)}$  ($K=1$, $a=1.5$, $b=0.5$, $\gamma_k=\gamma=0.01$, $T=0.1$). Left: CPU times vs $m$. Right: relative errors vs $m$.
}
\label{fig6.8}
\end{figure}

\begin{figure}
{ \centering
\includegraphics[width=0.48\textwidth]{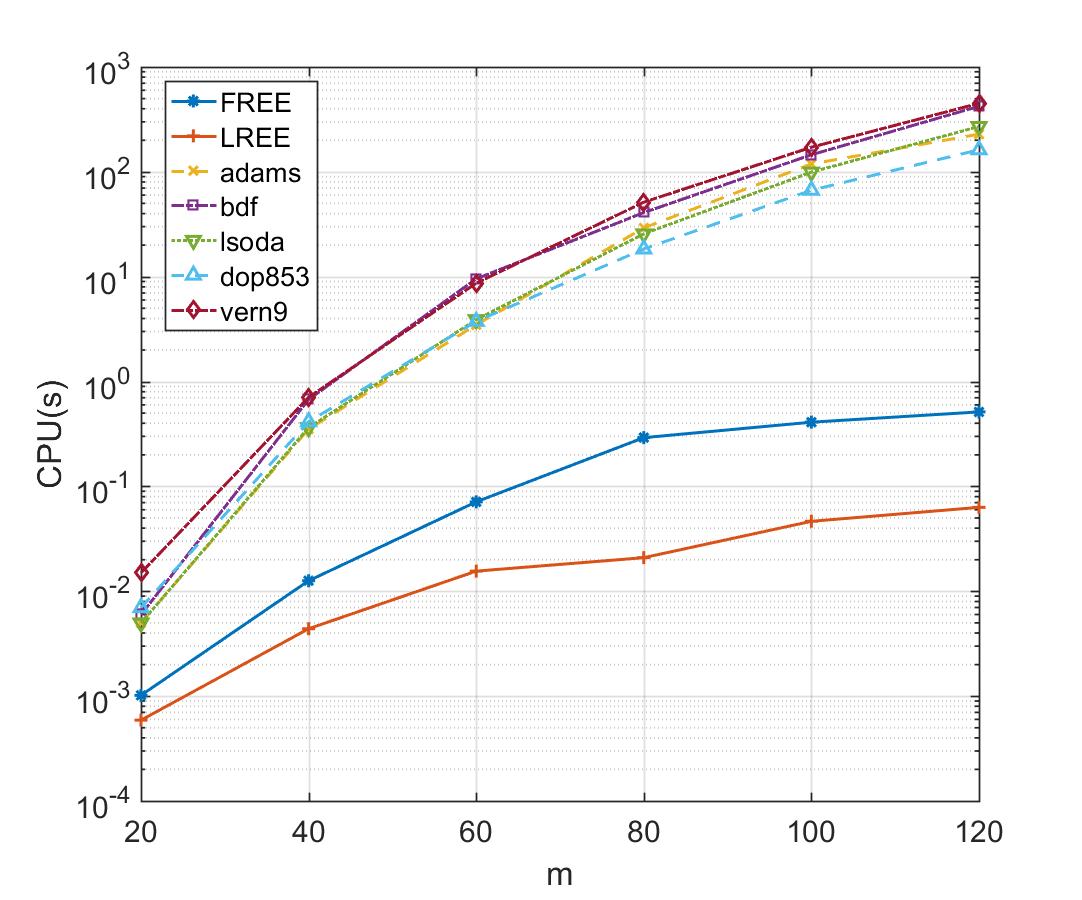}
} { \centering
\includegraphics[width=0.48\textwidth]{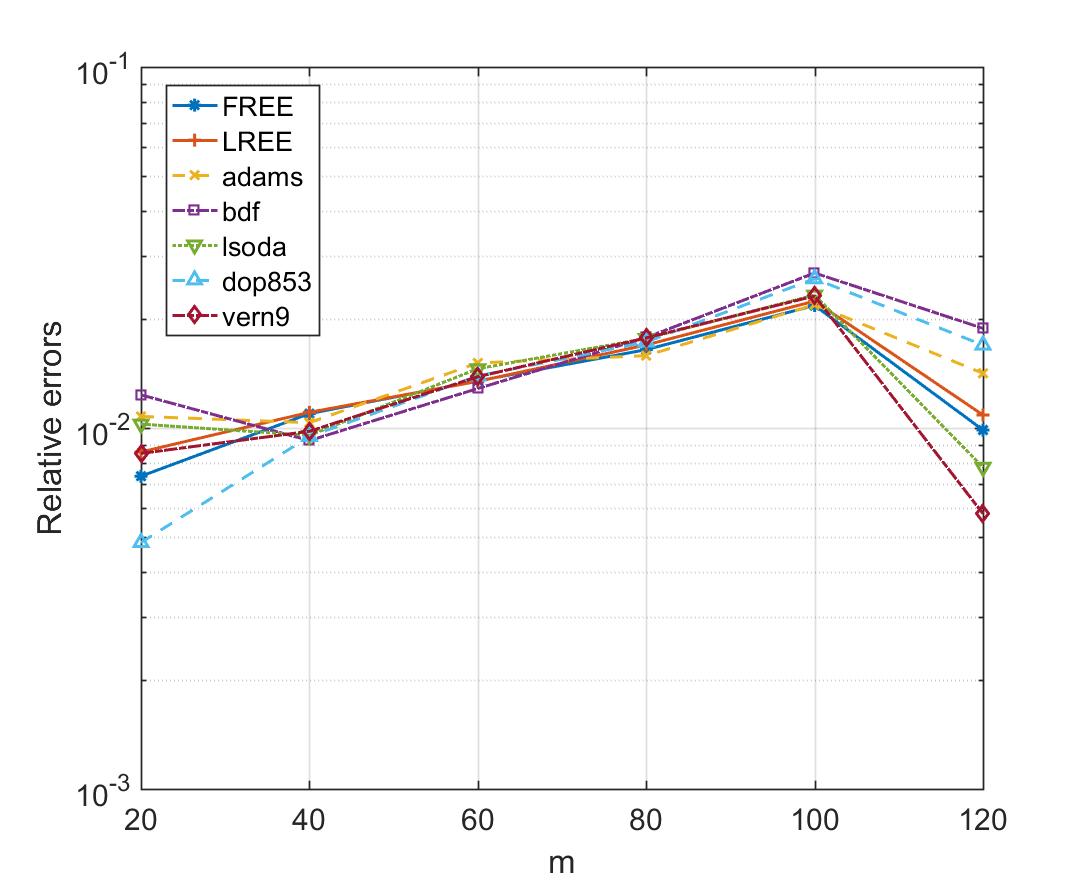}
} \caption{Numerical comparison between the proposed exponential schemes and the solvers in QuTip for the Lindblad equation with Hamiltonian \eqref{equ6.1} and random jump operators $L_k$ ($K=1$, $a=1.5$, $b=0.5$, $\gamma_k=\gamma=0.01$, $T=0.1$). Left: CPU times vs $m$. Right: relative errors vs $m$.
}
\label{fig6.9}
\end{figure}

In Figures \ref{fig6.10}-\ref{fig6.11}, we compare the performance of our exponential schemes with the five solvers in QuTip concerning the preservation of positivity and unit trace. We set
$tol_1=\frac{\tau}{10}$ and $tol_2=\frac{\tau^2}{10}$ for the LREE scheme and $\tau=0.1$ for both the FREE and LREE schemes. The absolute and relative tolerances of the \emph{adams}, \emph{bdf}, \emph{lsoda} and \emph{dop853} solvers are all set to be $10^{-3}$. The absolute and relative tolerances of the \emph{vern9} solver are set to be $10^{-2}$. We see that all the solvers can preserve unit trace.
However, we also see that, while both our FREE and LREE schemes are positivity preserving, the solvers in QuTip do not guarantee the positive semidefinite property of the density matrix, which is in agreement with the first-order barrier for
positivity of (standard) Runge-Kutta and multistep methods as discussed in \cite{Blanes2022,Martiradonna2020}.

\begin{figure}
{ \centering
\includegraphics[width=0.48\textwidth,height=0.23\textheight]{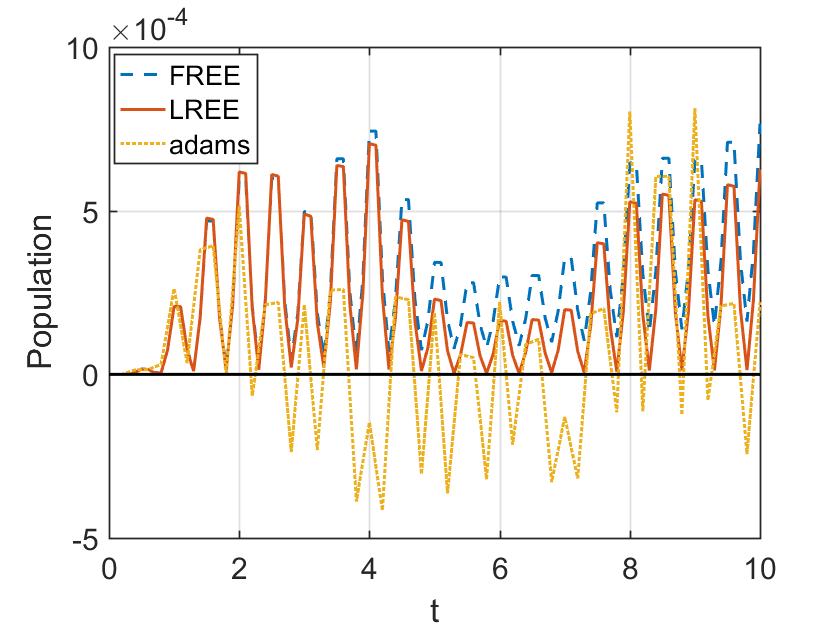}
} { \centering
\includegraphics[width=0.48\textwidth,height=0.23\textheight]{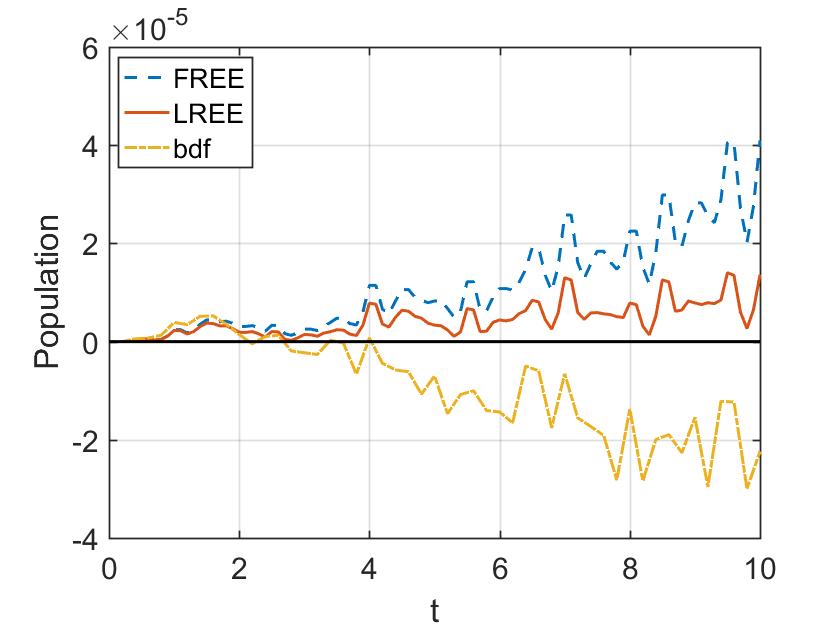}
}\\ { \centering
\includegraphics[width=0.48\textwidth,height=0.23\textheight]{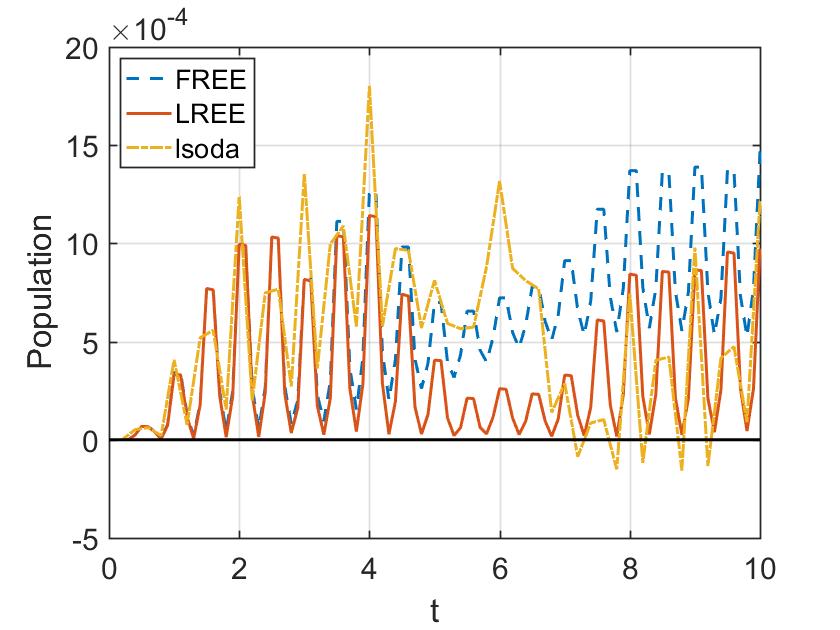}
} { \centering
\includegraphics[width=0.48\textwidth,height=0.23\textheight]{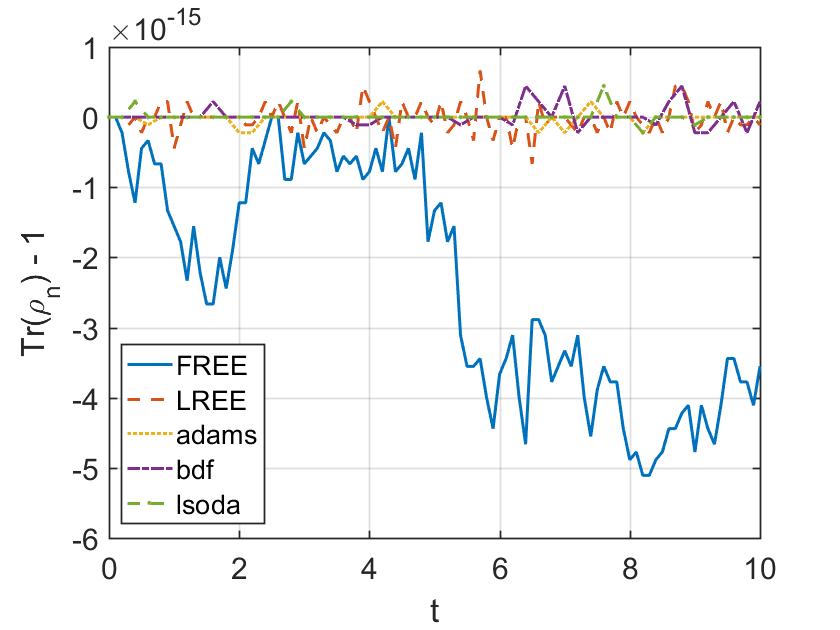}
}\caption{Numerical comparison between the proposed exponential schemes and \emph{adams}, \emph{bdf}, \emph{lsoda} solvers in QuTip for the Lindblad equation with Hamiltonian \eqref{equ6.1} and $d=4$, $K=3$, $a=1$, $b=1$, $\gamma_k=\gamma=0.05$, $g_{kl}(t)= \delta_{k,l-1}\cdot\sin(2\pi t)$. Top left: evolutions of the populations $\rho_{46,46}$. Top right:
evolutions of the populations $\rho_{45,45}$. Bottom left: evolutions of the populations $\rho_{38,38}$. Bottom right: evolutions of $\mbox{Tr}(\rho_n)-1$ vs time.}
\label{fig6.10}
\end{figure}

\begin{figure}
{ \centering
\includegraphics[width=0.48\textwidth]{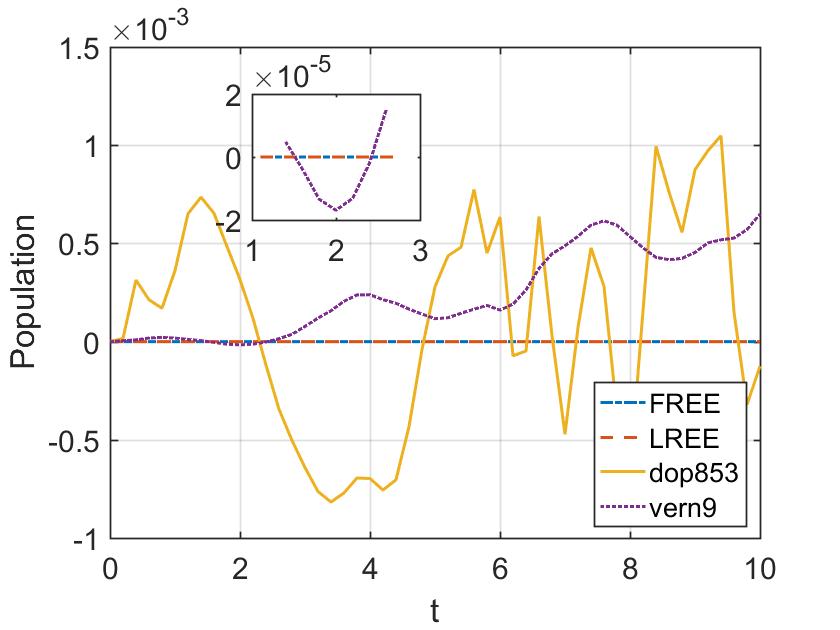}
} { \centering
\includegraphics[width=0.48\textwidth]{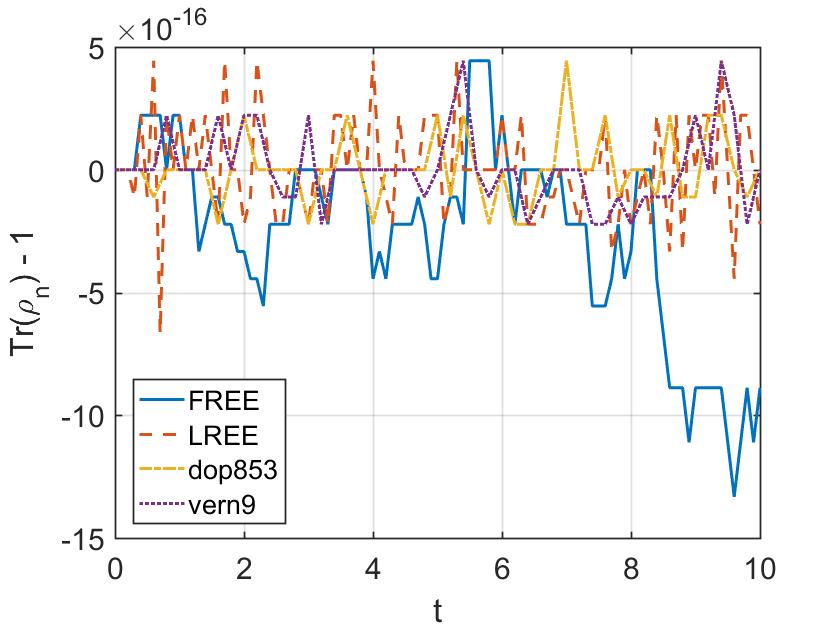}
} \caption{Numerical comparison between the proposed exponential schemes and \emph{dop853}, \emph{vern9 } solvers in QuTip for the Lindblad equation with Hamiltonian \eqref{equ6.1} and $d=4$, $K=3$, $a=1$, $b=1$, $\gamma_k=\gamma=0.05$, $g_{kl}(t)= \delta_{k,l-1}\cdot10\sin(10\pi t)$. Left: evolutions of the populations $\rho_{17,17}$. Right: evolutions of $\mbox{Tr}(\rho_n)-1$ vs time.
}
\label{fig6.11}
\end{figure}

\FloatBarrier

\section{Conclusion}
In this paper, full- and low-rank exponential Euler integrators for solving
the Lindblad equation were developed and analyzed. It was proven and numerically
validated that these schemes preserve positivity and unit trace unconditionally.
Moreover, sharp accuracy estimates were presented. The
computational performance of the new exponential Euler integrators were
successfully validated by results of numerical simulations with the Lindblad equation with various Hamiltonians, and by comparison with the Lindblad solvers available in the well-known QuTip package for simulation of open quantum systems.

\section*{Acknowledgements}
A.B. would like to thank Guofeng Zhang (Hong Kong Polytechnic University),
Xiu-Hao Deng
(Shenzhen Institute of Quantum Science and Engineering, SUSTech) and
Song Zhang (Shenzhen International Quantum Academy)
for inspiring discussions on open quantum systems and very kind hospitality.

\hspace*{\fill}

\noindent\small \textbf{Funding} ~~ The work of Hao Chen was partly supported by the Natural Science Foundation Project of CQ CSTC (No. cstc2021jcyj-msxmX0034).

\bibliographystyle{unsrt}
\bibliography{literature}

\begin{thebibliography}{10}

\bibitem{Breuer}
H.-P. Breuer and F.~Petruccione.
\newblock {\em {The Theory of Open Quantum Systems}}.
\newblock Oxford University Press, 2007.

\bibitem{Davies1976}
E.B. Davies.
\newblock {\em Quantum Theory of Open Systems}.
\newblock Academic Press, 1976.

\bibitem{Gorini}
V.~Gorini, A.~Kossakowski, and E.C.G. Sudarshan.
\newblock {Completely positive dynamical semigroups of N-level systems}.
\newblock {\em Journal of Mathematical Physics}, 17(5):821--825, 05 1976.

\bibitem{Lindblad}
G.~Lindblad.
\newblock On the generators of quantum dynamical semigroups.
\newblock {\em Communications in Mathematical Physics}, 48(2):119--130, 1976.

\bibitem{Riesch}
M.~Riesch and C.~Jirauschek.
\newblock Analyzing the positivity preservation of numerical methods for the
  {L}iouville-von {N}eumann equation.
\newblock {\em Journal of Computational Physics}, 390:290--296, 2019.

\bibitem{Ziolkowski}
R.W. Ziolkowski, J.M. Arnold, and D.M. Gogny.
\newblock Ultrafast pulse interactions with two-level atoms.
\newblock {\em Phys. Rev. A}, 52:3082--3094, Oct 1995.

\bibitem{Bidegaray}
B.~Bidégaray, A.~Bourgeade, and D.~Reignier.
\newblock Introducing physical relaxation terms in {B}loch equations.
\newblock {\em Journal of Computational Physics}, 170(2):603--613, 2001.

\bibitem{Blanes2022}
S.~Blanes, A.~Iserles, and S.~Macnamara.
\newblock Positivity-preserving methods for ordinary differential equations.
\newblock {\em ESAIM: M2AN}, 56(6):1843--1870, 2022.

\bibitem{Martiradonna2020}
A~Martiradonna, G.~Colonna, and F.~Diele.
\newblock Geco: Geometric conservative nonstandard schemes for biochemical
  systems.
\newblock {\em Applied Numerical Mathematics}, 155:38--57, 2020.
\newblock Structural Dynamical Systems: Computational Aspects held in Monopoli
  (Italy) on June 12-15, 2018.

\bibitem{Riesch1}
M.~Riesch, A.~Pikl, and C.~Jirauschek.
\newblock Completely positive trace preserving methods for the {L}indblad
  equation.
\newblock In {\em 2020 International Conference on Numerical Simulation of
  Optoelectronic Devices (NUSOD)}, pages 109--110, 2020.

\bibitem{Saut}
A.~Bourgeade and O.~Saut.
\newblock Numerical methods for the bidimensional {M}axwell-{B}loch equations
  in nonlinear crystals.
\newblock {\em Journal of Computational Physics}, 213(2):823--843, 2006.

\bibitem{Songolo1}
M.E. Songolo and B.~Bid\'{e}garay-Fesquet.
\newblock Strang splitting schemes for {N}-level {B}loch models.
\newblock {\em International Journal of Modeling, Simulation, and Scientific
  Computing}, 14(03):2350044, 2023.

\bibitem{Songolo2}
M.E. Songolo and B.~Bid\'{e}garay-Fesquet.
\newblock Nonstandard finite-difference schemes for the two-level {B}loch
  model.
\newblock {\em International Journal of Modeling, Simulation, and Scientific
  Computing}, 09(04):1850033, 2018.

\bibitem{LeBris1}
C.~Le~Bris and P.~Rouchon.
\newblock Low-rank numerical approximations for high-dimensional {L}indblad
  equations.
\newblock {\em Phys. Rev. A}, 87:022125, Feb 2013.

\bibitem{LeBris2}
C.~Le~Bris, P.~Rouchon, and J.~Roussel.
\newblock Adaptive low-rank approximation and denoised {M}onte {C}arlo approach
  for high-dimensional {L}indblad equations.
\newblock {\em Phys. Rev. A}, 92:062126, Dec 2015.

\bibitem{Cao}
Y.~Cao and J.~Lu.
\newblock Structure-preserving numerical schemes for {L}indblad equations,
  2024.

\bibitem{Schlimgen}
A.W. Schlimgen, K.~Head-Marsden, L.M. Sager, P.~Narang, and D.A. Mazziotti.
\newblock Quantum simulation of the {L}indblad equation using a unitary
  decomposition of operators.
\newblock {\em Phys. Rev. Res.}, 4:023216, Jun 2022.

\bibitem{Weimer}
H.~Weimer, A.~Kshetrimayum, and R.~Or\'us.
\newblock Simulation methods for open quantum many-body systems.
\newblock {\em Rev. Mod. Phys.}, 93:015008, Mar 2021.

\bibitem{Werner}
A.H. Werner, D.~Jaschke, P.~Silvi, M.~Kliesch, T.~Calarco, J.~Eisert, and
  S.~Montangero.
\newblock Positive tensor network approach for simulating open quantum
  many-body systems.
\newblock {\em Phys. Rev. Lett.}, 116:237201, Jun 2016.

\bibitem{Orus}
R.~Orús.
\newblock A practical introduction to tensor networks: Matrix product states
  and projected entangled pair states.
\newblock {\em Annals of Physics}, 349:117--158, 2014.

\bibitem{Schollwock}
U.~Schollw\"ock.
\newblock The density-matrix renormalization group.
\newblock {\em Rev. Mod. Phys.}, 77:259--315, Apr 2005.

\bibitem{Vidal1}
G.~Vidal.
\newblock Efficient classical simulation of slightly entangled quantum
  computations.
\newblock {\em Phys. Rev. Lett.}, 91:147902, Oct 2003.

\bibitem{Vidal2}
Guifr\'e Vidal.
\newblock Efficient simulation of one-dimensional quantum many-body systems.
\newblock {\em Phys. Rev. Lett.}, 93:040502, Jul 2004.

\bibitem{Hartmann}
M.J. Hartmann and G.~Carleo.
\newblock Neural-network approach to dissipative quantum many-body dynamics.
\newblock {\em Phys. Rev. Lett.}, 122:250502, Jun 2019.

\bibitem{Nagy}
A.~Nagy and V.~Savona.
\newblock Variational quantum {M}onte {C}arlo method with a neural-network
  ansatz for open quantum systems.
\newblock {\em Phys. Rev. Lett.}, 122:250501, Jun 2019.

\bibitem{Reh}
M.~Reh, M.~Schmitt, and M.~G\"arttner.
\newblock Time-dependent variational principle for open quantum systems with
  artificial neural networks.
\newblock {\em Phys. Rev. Lett.}, 127:230501, Dec 2021.

\bibitem{Arceci}
L.~Arceci, P.~Silvi, and S.~Montangero.
\newblock Entanglement of formation of mixed many-body quantum states via tree
  tensor operators.
\newblock {\em Phys. Rev. Lett.}, 128:040501, Jan 2022.

\bibitem{Sulz}
D.~Sulz, C.~Lubich, G.~Ceruti, I.~Lesanovsky, and F.~Carollo.
\newblock Numerical simulation of long-range open quantum many-body dynamics
  with tree tensor networks.
\newblock {\em Phys. Rev. A}, 109:022420, Feb 2024.

\bibitem{Chen1}
H.~Chen and A.~Borz{\`\i}.
\newblock Positivity preserving exponential integrators for differential
  {R}iccati equations.
\newblock {\em Journal of Scientific Computing}, 96(2):50, 2023.

\bibitem{Chen2}
H.~Chen and A.~Borz{\`\i}.
\newblock Low-rank exponential integrators for stiff differential {R}iccati
  equations.
\newblock {\em Submitted}, 2024.

\bibitem{Johansson}
J.R Johansson, P.D. Nation, and F.~Nori.
\newblock Qutip 2: A {P}ython framework for the dynamics of open quantum
  systems.
\newblock {\em Computer Physics Communications}, 184(4):1234--1240, 2013.

\bibitem{Hochbruck}
M.~Hochbruck and A.~Ostermann.
\newblock Exponential integrators.
\newblock {\em Acta Numerica}, 19:209--286, 2010.

\bibitem{Lang}
N.~Lang, H.~Mena, and J.~Saak.
\newblock On the benefits of the $ldl^t$ factorization for large-scale
  differential matrix equation solvers.
\newblock {\em Linear Algebra and its Applications}, 480:44--71, 2015.

\bibitem{Al-Mohy1}
A.H. Al-Mohy and N.J. Higham.
\newblock A new scaling and squaring algorithm for the matrix exponential.
\newblock {\em SIAM Journal on Matrix Analysis and Applications},
  31(3):970--989, 2010.

\bibitem{Al-Mohy2}
A.H. Al-Mohy and N.J. Higham.
\newblock Computing the action of the matrix exponential, with an application
  to exponential integrators.
\newblock {\em SIAM Journal on Scientific Computing}, 33(2):488--511, 2011.

\bibitem{Gajic}
Z.~Gajic and M.T.J. Qureshi.
\newblock {\em Lyapunov Matrix Equation in System Stability and Control}.
\newblock Dover Books on Engineering. Dover Publications, 2008.

\bibitem{Hamilton}
J.D. Hamilton.
\newblock {\em Time Series Analysis}.
\newblock Princeton University Press, 2020.

\bibitem{Hairer}
E.~Hairer, S.P. N{\o}rsett, and G.~Wanner.
\newblock {\em Solving Ordinary Differential Equations I: Nonstiff Problems}.
\newblock Springer Series in Computational Mathematics. Springer Berlin
  Heidelberg, 2008.

\bibitem{Verner}
J.~H. Verner.
\newblock Explicit {R}unge-{K}utta methods with estimates of the local
  truncation error.
\newblock {\em SIAM Journal on Numerical Analysis}, 15(4):772--790, 1978.

\end{thebibliography}

\end{document}